\documentclass[reqno,10pt, centertags,draft]{amsart}
\usepackage{amsmath,amsthm,amscd,amssymb,latexsym,upref}
\usepackage{mathrsfs}
\allowdisplaybreaks[2]

\makeatletter
\def\theequation{\@arabic\c@equation}

\newcommand{\gaD}{\gamma_{{}_D}}
\newcommand{\gaN}{\gamma_{{}_N}}
\newcommand{\aT}{\mathfrak{a}_\Theta}

\newcommand{\Mf}{\mathfrak{m}}
\newcommand{\Mas}{\operatorname{Mas}}

\newcommand{\card}{\operatorname{card}}

\newcommand{\bbR}{{\mathbb{R}}}

\newcommand{\bbC}{{\mathbb{C}}}

\newcommand{\cB}{{\mathcal B}}

\newcommand{\cG}{{\mathcal G}}
\newcommand{\cH}{{\mathcal H}}

\newcommand{\cK}{{\mathcal K}}
\newcommand{\cL}{{\mathcal L}}
\newcommand{\cM}{{\mathcal M}}

\newcommand{\cX}{{\mathcal X}}

\newcommand{\no}{\nonumber}
\newcommand{\lb}{\label}

\newcommand{\Om}{\Omega}
\newcommand{\dOm}{{\partial\Omega}}

\newcommand{\tr}{\operatorname{Tr}}
\newcommand{\gr}{\operatorname{Gr}}

\newcommand{\ran}{\text{\rm{ran}}}

\newcommand{\dom}{\text{\rm{dom}}}

\newcommand{\sgn}{\operatorname{sign}}

\numberwithin{equation}{section}

\newtheorem{theorem}{Theorem}[section]
\newtheorem{lemma}[theorem]{Lemma}

\newtheorem{corollary}[theorem]{Corollary}
\theoremstyle{definition}
\newtheorem{definition}[theorem]{Definition}

\newtheorem{hypothesis}[theorem]{Hypothesis}
\newtheorem{remark}[theorem]{Remark}

\newcommand{\beq}{\begin{equation}}
\newcommand{\enq}{\end{equation}}

\newcommand{\vertiii}[1]{{\left\vert\kern-0.25ex\left\vert\kern-0.25ex\left\vert #1
    \right\vert\kern-0.25ex\right\vert\kern-0.25ex\right\vert}}

        \newcommand{\Sp}{\operatorname{Sp}}
\begin{document}

\title[Hadamard-type formulas via the Maslov form]{Hadamard-type formulas via the Maslov form}

\allowdisplaybreaks

\author[Y. Latushkin]{Yuri Latushkin}
\address{Department of Mathematics,
The University of Missouri, Columbia, MO 65211, USA}
\thanks{Supported by the NSF grant  DMS-1067929, by the Research Board and Research Council of the University of Missouri, and by the Simons Foundation. The authors thank Margaret Beck, Graham Cox and Chris  Jones for many stimulating discussions and 
are grateful to Fritz Gesztesy for suggesting many useful references.}

\email{latushkiny@missouri.edu}
\author[A. Sukhtayev]{Alim Sukhtayev}
\address{Department of Mathematics, Indiana University Bloomington \\
Bloomington, IN 47405, USA}
\email{alimsukh@iu.edu}
%\urladdr{http://www.math.missouri.edu/personnel/faculty/latushkiny.html}
%\thanks{}

\dedicatory{To Jan Pr\"uss with best wishes on the occasion of his 65-th birthday}

\begin{abstract}
Given a star-shaped bounded Lipschitz domain $\Omega\subset\bbR^d$, we consider the  Schr\"odinger operator $L_\cG=-\Delta+V$ on $\Omega$ and its restrictions $L^{\Omega_t}_\cG$ on the subdomains $\Omega_t$, $t\in[0,1]$, obtained by shrinking $\Omega$ towards its center. We impose either the Dirichlet or quite general Robin-type boundary conditions determined by a subspace $\cG$ of the boundary
space $H^{1/2}(\partial\Omega)\times H^{-1/2}(\partial\Omega)$, and assume that the potential is smooth and takes values in the set of symmetric $(N\times N)$ matrices. Two main results are proved: First, for any $t_0\in(0,1]$ we give an asymptotic formula for the eigenvalues $\lambda(t)$ of the operator $L^{\Omega_t}_\cG$  as $t\to t_0$ up to quadratic terms, that is, we explicitly compute the first and second $t$-derivatives of the eigenvalues. This includes the case of the eigenvalues with arbitrary multiplicities.  Second, we compute the first derivative of the eigenvalues via the (Maslov) crossing form utilized in symplectic topology to define the Arnold-Maslov-Keller index of a path in the set of Lagrangian subspaces of the boundary space. The path is obtained by taking the Dirichlet and Neumann traces of the weak solutions of the eigenvalue problems for 
$L^{\Omega_t}_\cG$.
\end{abstract}

\date{\today}

\maketitle

%{\scriptsize{\tableofcontents}}
\normalsize

%%%%%%%%%%%%%%
%%%%%%%%%%%%%%
\section{Introduction}

The main objective of this paper is to establish connections between the Hadamard-type formulas for variation of the eigenvalues of differential operators on multidimensional domains when the domain is changing, and a quadratic form, the crossing (Maslov) form, used in symplectic topology to compute the Arnold-Maslov-Keller index.

The celebrated Hadamard Formula gives the rate of change of a simple eigenvalue of the Dirichlet Laplacian for a domain whose boundary changes in normal direction. The derivative was computed in some special cases by Rayleigh \cite[p.338, eqn.11]{R} and for general two dimensional domains by Hadamard \cite{H}, but the first rigorous proof was given by Garabedian and Schiffer \cite{GS}. During the last sixty years these results were developed in several directions most notably by computing the higher variations of the eigenvalues as well as variations of other quantities such as Green's functions, etc. We refer to an inspiring book \cite{Henry} and two excellent reviews, \cite{Gr} and \cite[Section 4]{BLC}, devoted to the variational and spectral aspects of the problem, respectively; one can find there further references. In addition, there is of course a vast body of work on dependence of the eigenvalues of domains both in multidimensional and one-dimensional cases; we mention here \cite{BW,KZ,KWZ1,KWZ2} and the classical exposition in \cite[Section VII.6.5]{Kato}.

Generalizing the known results in several directions, in the current paper we continue the work in \cite{DJ11} and \cite{CJLS14,CJM}, and study the variation with $t$ of the eigenvalues of the multidimensional Schr\"odinger operators $L^{\Omega_t}_\cG=L_\cG\big|_{L^2(\Omega_t)}$
with matrix valued potential on a family of bounded Lipschitz domains $\Omega_t$, $t\in[0,1]$, obtained by shrinking a given star-shaped domain $\Omega=\Omega_1$  towards its center; here $L_\cG=-\Delta+V$ and the boundary conditions are determined by a subspace $\cG$ of the boundary space $\cH=H^{1/2}(\partial\Omega)\times H^{-1/2}(\partial\Omega)$. We impose either the Dirichlet or very general Robin-type boundary conditions, and allow eigenvalues of arbitrary multiplicity. We derive an asymptotic formula for the eigenvalues up to quadratic terms (that is, we explicitly compute the first and second derivative of the eigenvalues with respect to the parameter $t$). A technically involved part of the paper, the asymptotic formula for the eigenvalues, resembles in spirit the classical perturbation formulas from \cite{Kato}
for analytic operator families but must be proved from scratch as the family of the Schr\"odinger operators is not analytic.

However, the most important and novel ingredient of the current paper is a connection of the Hadamard-type formulas for the derivatives of the eigenvalues and the Maslov index computed via the (Maslov) quadratic crossing form. The Maslov index $\Mas(\Upsilon,\cG)$ of a path $\{\Upsilon(t)\}_{t\in[0,1]}$ in the set of Lagrangian subspaces relative to a given Lagrangian subspace $\cG$ counts the number of conjugate points $t_0$ where $\Upsilon(t_0)\cap\cG\neq\{0\}$ and their multiplicities, see the general discussion in \cite{arnold67,Arn85,BF98,CLM94,F,MS,rs93,RoSa95,SW08} and the literature cited therein. For the problem at hand the subspaces $\Upsilon(t)$ are obtained by taking the boundary traces of weak solutions of the eigenvalue problems for the operators $L^{\Omega_t}_\cG$ while $\cG$ determines the boundary conditions. The multiplicity of a conjugate point is in fact the signature of a certain quadratic form, the Maslov crossing form $\mathfrak{m}_{t_0}$, defined on the intersection $\Upsilon(t_0)\cap\cG$. Recently,
a great deal of attention was given to the study of relations between  the eigenvalue counting functions for the operators $L^{\Omega_t}_\cG$ and the Maslov index, see \cite{CDB06,CDB09,CDB11,CJLS14,CJM,DP,DJ11,HS15,JLM13,LSS,PW}. Results of this type go back to the classical theorems by Sturm and Morse \cite{A01,M63} and to more recent theorems by Arnold \cite{arnold67,Arn85}, Bott \cite{B56}, Duistermaat \cite{D76}, Smale \cite{S65}, and Uhlenbeck \cite{U73}, see \cite{CJLS14} and \cite{LSS} for a more detailed historical account. A typical theorem from \cite{CJLS14,CJM,HS15,JLM13,LSS}
is a formula saying that the difference of the eigenvalue counting functions for the operators $L^{\Omega_{t_1}}_\cG$ and $L^{\Omega_{t_2}}_\cG$ is equal to the Maslov index of the path $\{\Upsilon(t)\}_{t\in[t_1,t_2]}$. In the current paper we establish an ``infinitesimal'' version of the difference formula by proving that the $t$-derivatives of the eigenvalues of the operators $L^{\Omega_t}_\cG$ at $t=t_0$ can be computed via the values of the Maslov form at the conjugate point $t_0$.

The paper is organized as follows. In Section \ref{sec:AMR} we formulate hypotheses and main results. In Section \ref{secMindFrL} we provide a brief tutorial on the Maslov index and recall several results from \cite{CJLS14} on computation of the (Maslov) crossing  form. The proofs of the main results are given in Section \ref{sec:mainres}.

%%%%%%%%%%%%%%
{\bf Notations.}  We denote by $\cB(\cX_1,\cX_2)$ the set of bounded linear operators from a Hilbert space $\cX_1$ into a Hilbert space $\cX_2$ (real or complex), and abbreviate this as $\cB(\cX)$ when $\cX=\cX_1=\cX_2$. We denote by $I_\cX$ the identity operator on $\cX$. Given two closed linear subspaces $\cL,\cM\subset\cX$, we denote by $\cL+\cM$ their (not necessarily direct) sum, by $\cL\dot{+}\cM$ their direct sum (which need not be orthogonal), and by $\cL\oplus\cM$ their orthogonal sum. For a linear operator $T$ on a Banach space $\cX$ we denote by $T^{-1}$ its (bounded) inverse, by $\ker (T)$ its null space, by $\ran (T)$ its range, by $\Sp(T)$ its spectrum, by $T^*$ its adjoint (or transpose, when the space is real), by $T|_\cL$ its restriction to a subspace $\cL\subset\cX$. If $\cX$ is a Banach space and $\cX^\ast$ is its adjoint then ${}_{\cX^*}\langle v,u\rangle_{\cX}$ denotes the action of a functional $v\in\cX^*$ on $u\in \cX$. 
For a bounded domain $\Omega\subset\bbR^d$ with the boundary $\partial\Omega$ we denote by $H^s(\Omega)$ and $H^s(\partial\Omega)$ the usual Sobolev spaces. We abbreviate 
\beq\lb{Abb12}\begin{split}
\langle g,f\rangle_{1/2}&={}_{H^{-1/2}(\dOm;\bbR^N)}\langle (g_n)_{n=1}^N,(f_n)_{n=1}^N\rangle_{H^{1/2}(\dOm;\bbR^N)}%\\&\qquad
=\sum_{n=1}^N{}_{H^{-1/2}(\dOm)}\langle g_n, f_n\rangle_{H^{1/2}(\dOm)},\\
\langle u,v\rangle_{L^2(\Omega)}&= \sum_{n=1}^N\int_\Om u_n(x)v_n(x)\,dx,\, u=(u_n)_{n=1}^N, v=(v_n)_{n=1}^N\in L^2(\Om;\bbR^N),
\end{split}
\end{equation}
for the scalar product in $L^2(\Omega)$. We let $\top$ denote the transposition and use ``dot'' for $t$-derivatives.
Throughout, we suppress vector notations by writing $L^2(\Om)$ instead of $L^2(\Om;\bbR^N)$, etc. For a vector-valued function $u=(u_n)_{n=1}^N:\Om\to\bbR^N$ we write $\nabla u$ applying the gradient to each component $u_n$ of $u$, and analogously write $\Delta u=(\Delta u_n)_{n=1}^N$, etc. Given two vector-valued functions $u=(u_n)_{n=1}^N$ and $v=(v_n)_{n=1}^N$, we write $uv=(u_nv_n)_{n=1}^N$ for their componentwise product. We often use ``$\cdot$'' to denote the scalar product in $\bbR^d$ and write 
\begin{equation}\label{Abb1}
\nabla u\cdot\nabla v=\sum_{n=1}^N\langle\nabla u_n,\nabla v_n\rangle_{\bbR^d},\, \langle \nabla u,\nabla v\rangle_{L^2}=\sum_{n=1}^N\int_{\Om}\langle\nabla u_n(x),\nabla v_n(x)\rangle_{\bbR^d}\,dx.
\end{equation} 
If $U=(U_n)_{n=1}^N$ and $V=(V_n)_{n=1}^N$ are vector-valued functions with components $U_n,V_n\in\bbR^d$ then we
write
\beq\lb{Abb2}
U\cdot V=\sum_{n=1}^N\langle U_n, V_n\rangle_{\bbR^d},\, \langle U,V\rangle_{L^2(\Om)}=\sum_{n=1}^N\int_{\Om}\langle U_n(x), V_n(x)\rangle_{\bbR^d}\,dx.
\enq
 We denote by $\cH=H^{1/2}(\dOm)\times H^{-1/2}(\dOm)$ the boundary space, by $\gaD$ the Dirichlet and by $\gaN=\nu\cdot\gaD\nabla$ the weak Neumann trace, and write $\tr u=(\gaD u,\gaN u)$. We denote by $\cK_{\lambda,t}$ the following set:
 \beq\begin{split}
 \cK_{\lambda,t}&=\big\{u\in H^1(\Om): \langle\nabla u,\nabla\Phi\rangle_{L^2(\Om)}
 \\&\hskip1cm+\langle t^2\big(V(tx)-\lambda\big)u,\Phi\rangle_{L^2(\Om)}=0\text{ for all $\Phi\in H^1_0(\Om)$}\big\}, \,\,\lambda\in\mathbb{R},t\in\Sigma=[\tau,1], \tau>0.
 \end{split}\enq
 For any $\tau\in(0,1]$ we define the rescaled trace map $\tr_t$ by the formula
 \beq
 \tr_tu=(\gaD u, t^{-1}\gaN u),\, u\in\dom(\gaN)\subset H^1(\Om),\, \, t\in\Sigma=[\tau,1].
 \enq

\section{Assumptions and main results}\label{sec:AMR}

Let $\Omega\subset\bbR^d$ be a star-shaped bounded domain with Lipshitz boundary $\partial\Omega$, and consider the following eigenvalue problem
\begin{align}\label{BVP1}
	-\Delta u&+V(x)u=\lambda u,\,x\in\Omega,\lambda\in\bbR,\,
	\tr u\in\cG.
\end{align}
Here $u:\Omega\to\bbR^N$ is a vector-valued function %\footnote{Throughout the paper we suppress vector notations as much as possible, and consistently write $L^2(\Omega)$ instead of $L^2(\Omega;\bbR^N)=\big(L^2(\Omega)\big)^N$, $C^0(\Omega)$ instead of $C^0(\Omega;\bbR^{N\times N})$ etc.} 
in the real Sobolev space $H^1(\Omega)=H^1(\Omega;\bbR^N)$, the potential $V=V(\cdot)\in C^0(\overline{\Omega};\bbR^{N\times N})$ is a continuous, matrix-valued function having symmetric values, $V(x)^\top=V(x)$, and $\Delta=\partial_{x_1}^2+\dots+\partial_{x_d}^2$ is the Laplacian. We denote by $\tr$ the trace operator acting from $H^1(\Om)$ into the boundary space $\cH$,
\begin{equation}\label{dfTr}
	\tr\colon H^1(\Omega)\to \cH, \,
	u\mapsto (\gaD u,\gaN u),\quad \cH=H^{1/2}(\partial\Omega)\times H^{-1/2}(\partial\Omega),
\end{equation}
where $\gaD u=u\big|_{\partial\Omega}$ is the Dirichlet trace and $\gaN u=\nu\cdot\nabla u\big|_{\partial\Omega}$ is the (weak) Neumann trace on $\Om$.
%In particular, $\gaN$ is an unbounded operator from $H^1(\Om)$ into $H^{-1/2}(\dOm)$ with domain $\dom(\gaN)=\{u\in H^1(\Om): \Delta u\in L^2(\Om)\}$.
Throughout, the Laplacian is understood in the weak sense, i.e. as a map from $H^1(\Om)$ into $H^{-1}(\Om)=\big(H^1_0(\Om)\big)^\ast$.

The boundary condition in \eqref{BVP1} is determined by a given closed linear subspace $\cG$ of the boundary space $\cH=H^{1/2}(\partial\Omega)\times H^{-1/2}(\partial\Omega)$. 
We will assume that $\cG$ is Lagrangian with respect to the symplectic form $\omega$ defined on $\cH$ by
\begin{equation}\label{dfnomega}
	\omega\big((f_1,g_1),(f_2,g_2)\big)=\langle g_2,f_1\rangle_{1/2}-\langle g_1,f_2\rangle_{1/2},
\end{equation} where $\langle g,f\rangle_{1/2}$ denotes the action of the functional $g\in H^{-1/2}(\partial\Omega)$ on the function $f\in H^{1/2}(\partial\Omega)$. 
Moreover, throughout the paper wee assume that either  $\cG=\cH_D$, the Dirichlet subspace, where we denote
		\[\cH_D=\big\{(0,g)\in\cH: g\in H^{-1/2}(\dOm)\big\},\] or $\cG$ is {\em Neumann-based}, that is, $\cG$ is the graph,
 \[\cG=\gr(\Theta)=\big\{
	(f,\Theta f)\in\cH: f\in H^{1/2}(\partial\Omega)\big\},\] of a compact, selfadjoint operator $\Theta\colon H^{1/2}(\partial\Omega)\to H^{-1/2}(\partial\Omega)$. 
	
Our main standing assumptions are summarized as follows:
\begin{hypothesis} \lb{h1}
We assume that:
\begin{itemize}\item[(i)]  $\Omega\subset{\bbR}^d$, $d\ge2$, is a nonempty, open, bounded, star-shaped, Lipschitz domain;
\item[(ii)] the subspace $\cG$ in \eqref{BVP1} is Lagrangian with respect to the symplectic form \eqref{dfnomega};
\item[(iii)] the subspace $\cG$ is either the Dirichlet subspace $\cG=\cH_D$, or a Neumann-based subspace;
 for the Neumann-based case we assume that
the compact operator $\Theta$ satisfying $\cG=\gr(\Theta)$ is associated with a closed, semibounded (and therefore symmetric) bilinear form  $\aT$ on $L^2(\dOm)$ with domain $\dom(\aT)=H^{1/2}(\dOm)\times H^{1/2}(\dOm)$ such that the following two conditions hold:
\begin{enumerate}\item[(a)] The form $\aT$  is $L^2(\dOm)$-semibounded from above by a positive constant $c_\Theta$ such that 
\[\aT(f,f)\le c_\Theta\|f\|_{L^2(\dOm)}^2\quad\text{ for all }\quad f\in H^{1/2}(\dOm),\]
\item[(b)] The form $\aT$  is $H^{1/2}(\dOm)$-bounded, that is, there is a positive constant $c_{\mathfrak{a}}$ such that
\[|\aT(f,f)|\le c_{\mathfrak{a}}\|f\|_{H^{1/2}(\dOm)}^2\quad\text{ for all }\quad f\in H^{1/2}(\dOm);\]
\end{enumerate}
\item[(iv)]  the potential $V$ is continuous, $V\in C^0(\overline{\Omega}; \bbR^N\times\bbR^N)$, and $V(x)$ is a symmetric $(N\times N)$ matrix for each $x\in\Omega$.
\end{itemize}
Sometimes assumption (i) will be replaced by a stronger assumption:
\begin{itemize}\item[(i')]  $\Omega\subset{\bbR}^d$, $d\ge2$, is a nonempty, open, bounded, star-shaped domain with $C^{1,r}$ boundary for some $1/2<r<1$.
\end{itemize}
When needed, for the operator $L^t_{\cG}$ defined below in \eqref{L_t}, we will assume that
\begin{itemize}\item[(iii')] assumption (iii) holds and $\dom(L^t_{\cG})\subset H^2(\Om)$.
\end{itemize}
When needed we will assume that
\begin{itemize}\item[(iv')] assumption (iv) holds and $V\in C^1(\overline{\Omega}; \bbR^N\times\bbR^N)$ is continuously differentiable.
\hfill$\Diamond$
\end{itemize}
\end{hypothesis}

Assuming Hypothesis \ref{h1} (iii), there exists a unique bounded, compact selfadjoint operator \\ ${\Theta\in\cB(H^{1/2}(\dOm), H^{-1/2}(\dOm))}$, defined for $f\in H^{1/2}(\dOm)$ by 
\beq\lb{qftheta}
\langle\Theta f,h\rangle_{1/2}=\aT(f,h)\,\text{ for }\, h\in H^{1/2}(\dOm),
\enq so that $\|\Theta\|_{\cB(H^{1/2}(\dOm), H^{-1/2}(\dOm))}\le c_{\mathfrak{a}}$ and the following assertions hold:
\begin{align}\lb{propTheta1}
\langle \Theta f,h\rangle_{1/2}&=
\langle \Theta h,f\rangle_{1/2}\quad\text{for all}\quad f,h\in H^{1/2}(\dOm),\\
\langle \Theta f,f\rangle_{1/2}&\le c_\Theta\|f\|_{L^2(\dOm)}^2
\quad\text{for all}\quad f\in H^{1/2}(\dOm).\lb{propTheta2}
\end{align}
 We note again that the operator $\Theta$ here and throughout is acting between the real Hilbert spaces $H^{1/2}(\dOm; \bbR^N)$ and $H^{-1/2}(\dOm; \bbR^N)$; it can be made complex by a standard procedure, cf. \cite[Section 5.5.3]{We80}, but this will not be needed until Section \ref{sec:mainres}.

Following the strategy of \cite{DJ11} and \cite{CJLS14}, a path of Lagrangian subspaces will be formed by shrinking the domain $\Om$ and rescaling the boundary value problem \eqref{BVP1} accordingly.
Since $\Omega$ is star-shaped, without loss of generality we assume that $0\in\Omega$, and for each $x\in\Omega$ there exist a unique $t\in(0,1]$ and $y\in\partial\Omega$ such that $x=ty$. For each $ t\in(0,1]$ we define a subdomain $\Om_t$ and the rescaled trace operator $\tr_t:H^1(\Omega)\to\cH$ by setting
\begin{equation}\label{defOmt}
	\Omega_t=\{x\in\Omega: x=t' y\,\text{ for }\, t'\in[0,t),\,y\in\dOm\},\,
	\tr_tu:= (\gaD  u,t^{-1}\gaN  u).
\end{equation}
Using the rescaled operator $\tr_t$ and the subspace $\cG$ from \eqref{BVP1}, we will consider the following family of boundary value problems on $\Om$, parametrized by $t$,
\begin{align}\label{tBVP1}
	-\Delta u+t^2V(tx)u=t^2\lambda u,\,x\in\Omega,\,\lambda\in\bbR,\, t\in(0,1],\, \tr_tu\in\cG.
\end{align}
 
 For a fixed $\lambda\in\mathbb{R}$, we let $\cK_{\lambda,t}$ denote the subspace in $H^1(\Om)$ of weak solutions to the equation $(-\Delta +t^2V(tx)-t^2\lambda)u=0$ for $t\in[\tau ,1]$, with $\tau \in(0,1]$ fixed.
Observe that no boundary conditions are imposed on the functions $u\in\cK_{\lambda, t}$. 

As we will see below, our hypotheses imply that the subspaces $\Upsilon(\lambda,t)=\tr_t(\cK_{\lambda,t})$ form a smooth path in the Fredholm--Lagrangian Grassmanian of the boundary space $\cG$. Therefore, for each $\lambda\in\bbR$ one can define the Maslov index of the path $t\mapsto\Upsilon(\lambda,t)$ with respect to $\cG$. Intuitively, this is a signed count of the crossings $t_0$ at which $\Upsilon(\lambda,t_0)$ and $\cG$ intersect nontrivially and the sign depends on the manner in which $\Upsilon(\lambda, t)$ passes through $\cG$ as $t$ increases. In particular, for each crossing $t_0$ a finite dimensional (Maslov) crossing form $\mathfrak{m}_{t_0}(p,q)$ is defined for $p,q\in\Upsilon(\lambda, t_0)\cap\cG$, cf. \eqref{dfnMasForm}, and the signature of the form is the sign of the crossing. We will see below that $\Upsilon(\lambda_0,t_0)\cap\cG\neq\{0\}$ if and only if $t_0^2\lambda_0$ is an eigenvalue of the boundary value problem \eqref{tBVP1} or, equivalently,  $\lambda_0$ is an eigenvalue of the boundary value problem \eqref{BVP1} restricted to $\Om_{t_0}$. Let $\lambda_0=\lambda(t_0)$ be an eigenvalue of multiplicity $m$ of \eqref{BVP1} restricted to $\Omega_{t_0}$. For $t$ near $t_0$ there will be $m$ eigenvalues $\lambda_j(t)$, $j=1,\dots,m$ of \eqref{BVP1} restricted to $\Om_t$,  bifurcating from $\lambda(t_0)$ (they may be repeating).
A main result of this paper is a formula for the $t$-derivative at a crossing of the eigenvalues $\lambda_j(t)$ of the boundary value problem \eqref{BVP1} restricted to $\Om_t$ expressed in terms of the Maslov crossing form $\mathfrak{m}_{t_0}(p,q)$ at $t_0$. To formulate the results, we will now introduce differential operators associated with \eqref{BVP1} and \eqref{tBVP1}.  

Let us consider first the following unitary operators:
\begin{align*}
U_t\colon &L^2(\Om_t)\to L^2(\Om),\quad (U_tw)(x)=t^{d/2}w(tx),\,x\in\Om,\\
U_t^\partial\colon &L^2(\dOm_t)\to L^2(\dOm),\quad (U_t^\partial h)(y)=t^{(d-1)/2}h(ty),\,y\in\dOm,\\
U_{1/t}^\partial\colon &L^2(\dOm)\to L^2(\dOm_t),\quad (U_{1/t}^\partial f)(z)=t^{-(d-1)/2}f(t^{-1}z),\,z\in\dOm_t,
\end{align*}
so that $(U_t^\partial)^*=U_{1/t}^\partial$ on $L^2$. These also define bounded operators on the appropriate Sobolev spaces, i.e., $U_t\in\cB(H^1(\Om_t),H^1(\Om))$ and 
$U_t^\partial\in\cB(H^{1/2}(\dOm_t),H^{1/2}(\dOm))$. The boundedness of the operator $U_t^\partial$ follows from the boundedness of $U_t$ and the boundedness of $\gaD$. Also, we define $U_t^\partial\colon H^{-1/2}(\dOm_t)\to H^{-1/2}(\dOm)$ by letting
\begin{equation}\lb{dfnUdt}
\langle U_t^\partial g, h\rangle_{1/2}=
{}_{H^{-1/2}(\dOm_t)}\langle  g,U_{1/t}^\partial h\rangle_{H^{1/2}(\dOm_t)},
\, h\in H^{1/2}(\dOm).
\end{equation}
For a subspace $\cG\subset\cH$ of the boundary space 
$\cH=H^{1/2}(\dOm)\times H^{-1/2}(\dOm)$ we let  $\cG_t\subset\cH_t$ be 
the subspace of  
$\cH_t=H^{1/2}(\dOm_t)\times H^{-1/2}(\dOm_t)$ defined by
\begin{equation}\label{defcGt}
\cG_t=U_{1/t}^\partial(\cG)=\big\{(U_{1/t}^\partial f, U_{1/t}^\partial g):\, (f,g)\in\cG\big\}.
\end{equation}
 Let us introduce the rescaled potentials $V^t$ and ${V}_{\lambda,t}$, 
\begin{align}
\begin{split}\label{eq2.13}
V^t(x)&=t^2V(tx),\, {V}_{\lambda,t}(x)=t^2V(tx)-t^2\lambda,\,\,\lambda\in\bbR,\, x\in\Om,\\
\end{split}
\end{align}
and the following family of operators $L_{\cG}^{\Om_t}$ corresponding to the boundary value problem \eqref{BVP1} but restricted to $\Om_t$:
\begin{align}\lb{L(Omt)}
\begin{split}
L_\cG^{\Om_t}u(x)&=-\Delta u(x)+V(x)u(x),\, x\in\Om_t,\,\,\,
L_{\cG}^{\Om_t} :\,\,  \dom(L_{\cG}^{\Om_t})\subset L^2(\Omega_t)\to L^2(\Omega_t),\\
 \dom(L_{\cG}^{\Om_t})&=\big\{u\in H^1(\Omega_t)\,|\,\Delta u\in L^2(\Omega_t) \,\text{ and }\\
&\gamma_{{}_{D,\dOm_t}} u=0 \,\,\hbox{in}\,\,H^{1/2}(\dOm_t)\,\,\text{if}\,\,\cG \,\text{is the Dirichlet subspace,}\,\,\text{or}\\ &(\gamma_{{}_{N,\dOm_t}}-U_{1/t}^\partial\circ\Theta\circ U_t^\partial\circ\gamma_{{}_{D,\dOm_t}})u=0\,\,\hbox{in}\,\,H^{-1/2}(\dOm_t)\,\,\text{if}\,\,\cG \,\text{is the Neumann-based subpace} \big\}.
\end{split}
\end{align}
In addition, we introduce operators $L^t_{\cG}$ obtained by rescaling $L_\cG^{\Om_t}$ from $\Om_t$ back to $\Om$ as follows:
\begin{align}
\begin{split}\lb{L_t}
L^t_{\cG}u(x)&=-\Delta u(x)+V^t(x)u(x),\, x\in\Om,\,\,\,L^t_{\cG}:\,\,\dom(L^t_{\cG})\subset L^2(\Omega)\to L^2(\Omega),\\
 \dom(L^t_{\cG})&=\big\{u\in H^1(\Omega)\,|\,\Delta u\in L^2(\Omega) \,\text{ and }\\
&\gaD u=0 \,\,\hbox{in}\,\,H^{1/2}(\dOm)\,\,\text{if}\,\,\cG \,\text{is the Dirichlet subpace,}\,\,\,\text{or}\\ &(\gaN-t\Theta\gaD)u=0\,\,\hbox{in}\,\,H^{-1/2}(\dOm)\,\,\text{if}\,\,\cG \,\text{is the Neumann-based subpace}  \big\}.
\end{split}
\end{align}
As we will see below, the operators $L^{\Om_t}_\cG$ and $L^t_\cG$ are selfadjoint with compact resolvent and therefore with only discrete spectrum. Moreover, $\lambda=\lambda(t)$ is an eigenvalue of $L_{\cG}^{\Om_{t}}$ acting on $L^2(\Om_t)$ if and only if the number $t^2\lambda(t)$ is an eigenvalue of $L^{t}_{\cG}$ acting on $L^2(\Om)$; we will use notation $\lambda^t$ for the latter thus setting $\lambda^t:=t^2\lambda(t)$. 

Let us suppose that $t_0\in(0,1]$ and $\lambda(t_0)$ is a fixed eigenvalue of the operator $L^{\Omega_{t_0}}_\cG$ of multiplicity $m$. Our objective is to describe the behavior of the eigenvalues $\{\lambda_{j}(t)\}_{j=1}^m$ of the operator $L_{\cG}^{\Om_t}$ for $t$ near $t_0$.
We note that $\lambda^{t_0}=t_0^2\lambda(t_0)$ is an eigenvalue of $L_{\cG}^{t_0}$ of multiplicity $m$, and denote by $P$ the Riesz projection in $L^2(\Om)$ for the operator $L_{\cG}^{t_0}$ onto $\ker(L_{\cG}^{t_0}-\lambda^{t_0}I_{L^2(\Om)})$ and by $S$ the reduced resolvent of the operator $L_{\cG}^{t_0}$, see formula \eqref{decomp1} below. In order to build the perturbation theory for the spectrum of $L_{\cG}^{t_0}$, we will need to introduce the following $m$-dimensional operators $T^{(1)}$ and $T^{(2)}$ acting in the subspace  $\ran(P)$, 
\begin{align}\lb{T1}
T^{(1)}&=P\big(\dot{V}^t\big|_{t=t_0}-\Theta_D\big)  P,\\
T^{(2)}&=P\big(\frac{1}{2}\ddot{V}^t\big|_{t=t_0}-\dot{V}^t\big|_{t=t_0}S\dot{V}^t\big|_{t=t_0}- \Theta_D   S\Theta_D  
+\dot{V}^t\big|_{t=t_0} S\Theta_D  + \Theta_D S\dot{V}^t\big|_{t=t_0}\big)P.\lb{newT2}
\end{align}
Here and throughout we use ``dot'' to denote the $t$-derivative and introduce notation, cf.\ Remark \ref{rem4.6},
\begin{equation*}
\Theta_D=\gaD^*\Theta\gaD\in\cB(H^{1}(\Omega),(H^{1}(\Omega))^*).
\end{equation*}
Let us consider a normalized basis $\{u_j^{t_0}\}_{j=1}^m$ of the subspace $\ran(P)$ consisting of the eigenvectors of the operator $T^{(1)}$ defined in \eqref{T1} and corresponding to its eigenvalues $\lambda^{(1)}_j$ so that $T^{(1)}u_j^{t_0}=\lambda_j^{(1)}u_j^{t_0}$, $j=1,\ldots,m$. We denote by $q_j=(
		\gaD u_j^{t_0}, \frac{1}{t_0}\gaN u_j^{t_0})\in\Upsilon(\lambda(t_0), t_0)\cap\cG$ the respective rescaled traces.
%\bbox{A small addition suggested by the referee:}

%\bbox{End of the addition}

We are now ready to formulate our main results. We begin with the first derivative of the eigenvalues.
\begin{theorem}[Hadamard-type formula via the Maslov form]\lb{gH}
Assume Hypothesis \ref{h1} (i), (ii), (iii) and (iv') and fix $t_0\in(0,1]$. Let $\lambda(t_0)$ be a given eigenvalue of multiplicity $m$ of the operator $L_{\cG}^{\Om_{t_0}}$ defined in \eqref{L(Omt)}, and let $\{\lambda_{j}(t)\}_{j=1}^m$ denote the eigenvalues of the operator $L_{\cG}^{\Om_t}$ for $t$ near $t_0$.  Then the rates of change of $\lambda_j(t)$ are given by the formulas
\begin{equation}\lb{1.17}
			\frac{d\lambda_j(t)}{dt}\Big|_{t=t_0}=\frac{1}{t_0^2}\big(\lambda_j^{(1)}-2t_0\lambda(t_0)\big)=\frac1{t_0}\mathfrak{m}_{t_0}(q_j,q_j),\,\,j=1,\ldots,m,
			\end{equation}
		where the (Maslov) crossing form for the path $t\mapsto\Upsilon(\lambda(t_0), t)=\tr_t\big(\cK_{\lambda(t_0),t}\big)$ at $t_0$ can be computed as follows:
	\begin{equation}\lb{1.18}
	\mathfrak{m}_{t_0}(q_j,q_j)=\frac1{t_0}\langle \dot{{V}}_{\lambda(t_0),t}\big|_{t=t_0}
	u_j^{t_0}, u_j^{t_0}\rangle_{L^2(\Om)}-\frac1{t^2_0}\langle\gaN  u_j^{t_0},\gaD   u_j^{t_0}\rangle_{1/2}.
	\end{equation}
	Here, $\lambda_j^{(1)}$, $q_j$ and $u^{t_0}_j$ are defined in the paragraph preceding the theorem and $V_{\lambda,t}$ is defined in \eqref{eq2.13}.
	\end{theorem}

We recall that $\dom (L^t_{\cG})\subset H^2(\Om)$ provided assumption (iii') holds. If this is the case then the strong Neumann trace $\gaN^su_j^{t_0}$ is defined. 
\begin{theorem}\lb{1.3} Assume Hypothesis \ref{h1} (i'), (ii), (iii') and  (iv') and fix $t_0\in(0,1]$. Let $\lambda_j(t)$, $u_j^{t_0}$, $q_j$, $j=1,\dots,m$, be as in Theorem \ref{gH}. %so that $T_{s_0}(\cK_{s_0})\cap\cG\neq\{0\}$. 
Then
\beq\lb{henF1}\begin{split}
\frac{d\lambda_j(t)}{dt}\Big|_{t=t_0}&=\frac1{t_0}\mathfrak{m}_{t_0}(q_j,q_j)=\frac1{(t_0)^3}\int_{\dOm}\Big(\big(\nabla u_j^{t_0}\cdot \nabla u_j^{t_0})(\nu\cdot x)-2\langle\nabla u_j^{t_0}\cdot x,\gaN^{\rm s}u_j^{t_0}(x)\rangle_{\bbR^N}\\&+(1-d)\langle\gaN^{\rm s} u_j^{t_0}(x), u_j^{t_0}(x)\rangle_{\bbR^N}+\langle (V(t_0x)-\lambda(t_0))u_j^{t_0}(x),u_j^{t_0}(x)\rangle_{\bbR^N}(\nu\cdot x)\Big)\,dx.
\end{split}\enq
\end{theorem} 

\begin{corollary}[Hadamard-type formula for the Dirichlet case]\lb{1.4}
	Assume Hypothesis \ref{h1} with (i') and (iv'). Let $\cG$ be the Dirichlet subspace of the boundary space $\cH=H^{1/2}(\dOm)\times H^{-1/2}(\dOm)$, and $\lambda_j(t), u_j^{t_0},q_j$, $j=1,\dots,m$, be as in Theorem \ref{gH}.  Then,
	\begin{equation}
	\frac{d\lambda_j(t)}{dt}\Big|_{t=t_0}=\frac1{t_0}\mathfrak{m}_{t_0}(q_j,q_j)=-\frac1{(t_0)^3}\int_{\dOm}\|\gaN^{\rm s} u_j^{t_0}\|_{\bbR^N}^2(\nu\cdot x)\,dx<0.
	\end{equation}
\end{corollary}

We will now proceed with the asymptotic formula up to quadratic terms for the eigenvalues $\lambda_j(t)$ for $t$ near $t_0$. To this end, we will renumber the eigenvalues of the operator $T^{(1)}$, and let $\{\lambda^{(1)}_i\}_{i=1}^{m'}$ denote the 
$m'$ {\em distinct} eigenvalues of the operator $T^{(1)}$, so that $m'\le m$. We 
let $m_i^{(1)}$ denote their multiplicities, and let $P^{(1)}_i$ 
denote the orthogonal Riesz spectral projections of $T^{(1)}$ corresponding to the eigenvalue $\lambda_i^{(1)}$.
Next, let $\lambda^{(2)}_{i k}$, $i=1,\dots, m'$, $k=1,\dots,m^{(1)}_i$,  denote the eigenvalues of the operator $P^{(1)}_i T^{(2)} P^{(1)}_i$ 
in $\ran(P^{(1)}_i)$ where the operator $T^{(2)}$ is defined in \eqref{newT2}.

\begin{theorem}[Asymptotic formula]\lb{gHaf}
Assume Hypothesis \ref{h1} (i), (ii), (iii) and (iv') and fix $t_0\in(0,1]$.  Let $\lambda(t_0)$ be a given eigenvalue of multiplicity $m$ of the operator $L_{\cG}^{\Om_{t_0}}$ defined in \eqref{L(Omt)}. Then the following asymptotic formula holds for the eigenvalues $\{\lambda_j(t)\}_{j=1}^m=\{\lambda_{ik}(t)\}$ of the operator $L_{\cG}^{\Om_t}$ as $t\to t_0$:
\begin{equation}\label{asform}
\begin{split}
\lambda_{ik}(t)&=\lambda(t_0)+\big(\frac{1}{t_0^2}\lambda_i^{(1)}
-\frac{2}{t_0}\lambda(t_0)\big)(t-t_0)\\
&+\big(\frac{1}{t_0^2}\lambda_{ik}^{(2)}-\frac{2}{t_0^3}\lambda_i^{(1)}+
\frac{3}{t_0^2}\lambda(t_0)\big)(t-t_0)^2+\mathrm{o}(t-t_0)^2,\,
i=1,\dots, m',\, k=1,\dots, m_i^{(1)}.
\end{split}
\end{equation} 
\end{theorem}

%%%%%%%%%%%%%%

%%%%%%%%%%%%%%%%%%%%%%%%%%%%%%%%%%

%%%%%%%%%%%%%%
%%%%%%%%%%%%%%
\section{Preliminaries: the Maslov index and symplectic view on eigenvalue problems}\label{secMindFrL}

In this section we collect some preliminary results relevant to the Maslov index and the (Maslov) crossing form. Also, we describe a symplectic approach to the eigenvalue problems for the operators $L^{\Omega_t}_\cG$ used in \cite{DJ11} and \cite{CJLS14, CJM}. Finally, we recall some formulas for the crossing form obtained in  \cite{CJLS14}.

\subsection{The Maslov index}
Let $\cX$ be a real Hilbert space equipped with a symplectic form $\omega$, that is, a bounded, skew-symmetric, nondegenerate bilinear form. Let  $J \in \cB(\cX)$ be the associated complex structure, which satisfies $\omega(u,v)=\langle Ju,v\rangle_\cX$ for the scalar product in $\cX$, $J^2=-I_\cX$ and $J^*=-J$.
\begin{definition}\label{dfnFLG}
 {\bf (i)}\, We say that two closed linear subspaces $\cK,\cL$ of $\cX$ form a {\em Fredholm pair} if their intersection $\cK\cap\cL$ has finite dimension and their sum $\cK+\cL$ has finite codimension (and hence is closed \cite[Section IV.4.1]{Kato}). Given a closed linear subspace $\cK$ of $\cX$, we define the {\em Fredholm Grassmannian
$F(\cK)$ of $\cK$} as the following set of closed linear subspaces of $\cX$:
%\beq\label{defFG}
$F(\cK)=\big\{\cL: \cL \text{ is closed and $(\cK,\cL)$ is a Fredholm pair}\}$.% \big\}.\enq
%and the {\em reduced Fredholm Grassmannian
%$\Fr(\cK)$ of $\cK$} as the following set of closed linear subspaces of $\cX$:
%\begin{equation}\label{defFr}\begin{split}
%\Fr(\cK)=\big\{A(\cK)&: A\in\cB(\cX) \text{ is boundedly invertible }\\&
%\hskip1cm  \text{ and } A-I_\cX\in\cB_\infty(\cX)\big\}.\end{split}\end{equation}
%
%\begin{equation}\label{defFr}
%\Fr(\cK)=\big\{A(\cK): A\in\cB(\cX) \text{ is boundedly invertible and } A-I_\cX\in\cB_\infty(\cX)\big\}.\end{equation}

{\bf (ii)}\, A closed linear subspace $\cK$ of $\cX$ is called {\em Lagrangian} if the form $\omega$ vanishes on $\cK$, that is, $\omega(u,v)=0$ for all $u,v\in\cK$, and $\cK$ is maximal, that is, 
 if $u\in\cX$ and $\omega(u,v)=0$ for all $v\in\cK$, then  $u\in\cK$. We denote by $\Lambda(\cX)$ the set of all Lagrangian subspaces in $\cX$.
Given $\cK\in\Lambda(\cX)$, we define the {\em Fredholm--Lagrangian Grassmannian $F\Lambda(\cK)$} to be the set of Lagrangian subspaces $\cL \subset \cX$ such that $(\cK,\cL)$ is a Fredholm pair; in other words
%\beq\label{defFLG}
$F\Lambda(\cK)=F(\cK)\cap\Lambda(\cX)$.%\enq
\hfill$\Diamond$
\end{definition}

We will now recall the definition of the Maslov index of a piece-wise smooth path $\Upsilon\colon\Sigma\to F\Lambda(\cG)$ with respect to a fixed Lagarangian subspace $\cG\in\Lambda(\cX)$, see Definition \ref{dfnDef3.6} below. 
We refer to \cite[Theorem 3.6]{F} for a list of basic properties of the Maslov index; in particular, the Maslov index is a homotopy invariant for homotopies keeping the endpoints fixed, and is additive under catenation of paths. 
We begin with 
the following elementary fact  needed to define the Maslov crossing form, see \cite[Lemma 3.8]{CJLS14} or \cite[Lemma 2.22]{F}.
Let  $\Sigma=[a,b]\subset\bbR$ be a set of parameters.
\begin{lemma}\lb{crossex}\cite{CJLS14}
Let $\{\Pi_s\}_{s\in\Sigma}$ be a family of orthogonal projections on $\cX$ such that the function $s\mapsto\Pi_s$ is in $C^k(\Sigma;\cB(\cX))$ for some $k\in\{0,1,\dots\}$. Then for any $s_0\in\Sigma$ there exists a neighborhood $\Sigma_0$ in $\Sigma$ containing $s_0$ and a family of operators $\{R_s\}$ from $\ran(\Pi_{s_0})$ into $\ker(\Pi_{s_0})$ such that 
the function $s\mapsto R_s$ is in $C^k(\Sigma_0;\cB(\ran(\Pi_{s_0}),\ker(\Pi_{s_0})))$ and for all $s\in\Sigma_0$, using the decomposition 
$\cX=\ran(\Pi_{s_0})\oplus\ker(\Pi_{s_0})$, we have
$\ran(\Pi_s)=\gr(R_s)=\{q+R_sq: q\in\ran(\Pi_{s_0})\}$.
Moreover,
$R_s\to0$ in $\cB(\ran(\Pi_{s_0}),\ker(\Pi_{s_0}))$ as $s\to s_0$.
\end{lemma}

We will now define the (Maslov) crossing form for a smooth path in the Fredholm--Lagrangian Grassmannian. Consider a $C^1$ path $\Upsilon\colon\Sigma\to F\Lambda(\cG)$, that is, a family $\{\Upsilon(s)\}_{s\in\Sigma}$
of Lagrangian subspaces such that the pair $(\cG,\Upsilon(s))$ is Fredholm for each $s\in\Sigma$ and the function $s\mapsto\Pi_s$ is in $C^1(\Sigma;\cB(\cX))$, where $\Pi_s$ denotes the orthogonal projection in $\cX$ onto the subspace $\Upsilon(s)$. Fix any $s_0\in\Sigma$ and use Lemma \ref{crossex} to find a neighborhood $\Sigma_0$ in $\Sigma$ containing $s_0$ and a $C^1$-smooth family of operators $R_s$ acting from $\Upsilon(s_0)=\ran(\Pi_{s_0})$ into $\ker(\Pi_{s_0})$ such that for all $s\in\Sigma_0$, using the decomposition $\cX=\ran(\Pi_{s_0})\oplus\ker(\Pi_{s_0})$, we have
\beq\lb{eqgr}\begin{split}
\Upsilon(s)=\ran(\Pi_s)=\gr(R_s)=\{q+R_sq:
q\in\ran(\Pi_{s_0})\}.\end{split}
\enq
\begin{definition}\lb{dfnMasF}
(i)\, We call $s_0\in\Sigma$ a {\em conjugate time} or {\em crossing} if $\Upsilon(s_0)\cap\cG\neq\{0\}$.

(ii)\, The finite-dimensional, symmetric bilinear form
\beq\lb{dfnMasForm}
{\Mf}_{s_0}(q,p)=\frac{d}{ds}\omega(q,R_sp)|_{s=s_0}=
\omega(q,\dot{R}(s_0)p)\,\text{ for }\, q,p\in\Upsilon(s_0)\cap\cG
\enq is called the {\em (Maslov) crossing form at $s_0$}.

(iii)\, The crossing $s_0$ is called {\em regular} if the crossing form $\Mf_{s_0}$ is nondegenerate;
it is called {\em positive} if the form is positive definite and {\em negative} if the form is negative definite.
\end{definition}
\begin{remark}\lb{rem:propcf}
The crossing form ${\Mf}_{s_0}$ in Definition \ref{dfnMasF} (ii) is finite dimensional since the pair of subspaces $(\cG,\Upsilon(s_0))$ is Fredholm. The form is symmetric since the subspace $\Upsilon(s)=\ran(\Pi_s)=\gr(R_s)$ is Lagrangian and thus the equality $\omega(q+R_sq,p+R_sp)=0$ holds for all $p,q\in\ran(\Pi_0)$. As any symmetric form, 
${\Mf}_{s_0}$ can be diagonalized; we will denote by $n_+({\Mf}_{s_0})$,
respectively, $n_-({\Mf}_{s_0})$ the number of positive, respectively negative squares of ${\Mf}_{s_0}$ and by $\sgn({\Mf}_{s_0})=n_+({\Mf}_{s_0})-n_-({\Mf}_{s_0})$ its signature. It can be shown, see e.g. \cite[Proposition 3.26]{F}, that the subspace $\ker(\Pi_{s_0})$ used in \eqref{eqgr} to construct the crossing form can be replaced by any subspace $\widetilde{\Upsilon}(s_0)$ of $\cX$ such that $R_s\in\cB(\Upsilon(s_0),\widetilde{\Upsilon}(s_0))$ and $\cX=\Upsilon(s_0)+\widetilde{\Upsilon}(s_0)$. Here the sum is not necessarily orthogonal, or even direct. The crossing form then does not depend on the choice of the subspace $\widetilde{\Upsilon}(s_0)$.\hfill$\Diamond$\end{remark}

\begin{definition}\label{dfnDef3.6}   Let $\cG$ be a Lagrangian subspace in a real Hilbert space $\cX$ and let $\Upsilon\colon\Sigma=[a,b]\to F\Lambda(\cG)$ be a $C^1$ smooth path in the Fredholm--Lagrangian Grassmannian. Let $s_0$ be the only regular crossing in a segment $\Sigma_0=[a_0,b_0]\subset\Sigma$ then the Maslov index $\Mas(\Upsilon|_{\Sigma_0},\cG)$ is defined as follows:
\beq\lb{MIcomp}
\Mas(\Upsilon|_{\Sigma_0},\cG)=\begin{cases}\sgn({\Mf}_{s_0})=n_+({\Mf}_{s_0})-n_-({\Mf}_{s_0})& \text{ if } s_0\in(a_0,b_0),\\
-n_-({\Mf}_{s_0})& \text{ if } s_0=a_0,\\
n_+({\Mf}_{s_0})& \text{ if } s_0=b_0.
\end{cases}
\enq
The crossing form can be used to define the Maslov index of a piecewise $C^1$-smooth path by computing each segment individually and summing.\hfill$\Diamond$
\end{definition}

%%%%%%%%%%%%%%%%%%%%%%%%%%%%%%%%%%%%%%%%%%
%\section{A symplectic view of the eigenvalue problem}\label{sec4}
\subsection{A symplectic view on the eigenvalue problems}
We will now recall from \cite{DJ11} and \cite{CJLS14} a way to describe the eigenvalue problem \eqref{BVP1} in terms of the intersections of a path of Lagrangian subspaces with a fixed subspace $\cG$. The path is formed by transforming the boundary value problems from the shrunken domain $\Om_t$ back to the original domain $\Omega$, and taking boundary traces of the weak solutions to the rescaled eigenvalue equations. We will begin with a simple rescaling.

%\subsection{Rescaling and the related selfadjoint differential operators}\label{sub4.1}

%Recalling the definition of the rescaled trace map $\tr_tu = (\gaD  u,t^{-1}\gaN  u)$ from \eqref{defOmt}, and defining the Dirichlet and Neumann trace maps $\gamma_{{}_{D,\dOm_t}}$ and $\gamma_{{}_{N,\dOm_t}}$ on  $\dOm_t$, we have the following result.
\begin{lemma}\label{lemResc}\cite{CJLS14} Assume Hypothesis \ref{h1}(i). Let $\cG$  be either the Dirichlet subspace or a Neumann-based subspace of $\cH$ and let $\cG_t$ be the subspace defined in \eqref{defcGt} for some $t\in(0,1]$.
Then a function $w\in H^1(\Om_t)$ solves the eigenvalue problem 
$L_{\cG}^{\Om_t}w=\lambda w$, $x\in\Omega_t$,
that is,
\begin{align}
-\Delta w+V(x)w(x)=\lambda w(x),\,x\in\Om_t,\,\,\tr_{\dOm_t} w=(\gamma_{{}_{D,\dOm_t}}w,\gamma_{{}_{N,\dOm_t}}w)\in\cG_t,
\label{BVPomt2}
\end{align} 
if and only if the function $u=U_tw\in H^1(\Om)$ solves the eigenvalue problem 
$L^t_{\cG}u=t^2\lambda u$, $x\in\Omega$,
that is,
\begin{align}
-\Delta u +V^t(x)u(x)=t^2\lambda u(x),\,x\in\Om,\,\,\tr_t u&=(\gaD u,t^{-1}\gaN u)\in\cG.
\label{BVPom2}\end{align}
Moreover, the multiplicity of $\lambda\in\Sp(L_{\cG}^{\Om_t})$ is the same as the multiplicity of $t^2\lambda\in\Sp(L^t_{\cG})$.
\end{lemma}

%%%%%%%%%%%%%%
We will now define a path $\Upsilon$ in the set of Lagrangian subspaces in $H^{1/2}(\dOm)\times H^{-1/2}(\dOm)$ by taking traces of weak solutions to the rescaled equation \eqref{BVPom2} introduced in Lemma \ref{lemResc}. We denote by $\cK_{\lambda,t}$ the following set,
\beq\lb{dfnKS}\begin{split}
\cK_{\lambda,t}&=\big\{u\in H^1(\Om): \langle\nabla u,\nabla\Phi\rangle_{L^2(\Om)}
\\&\hskip1cm+\langle t^2\big(V(tx)-\lambda\big)u,\Phi\rangle_{L^2(\Om)}=0\text{ for all $\Phi\in H^1_0(\Om)$}\big\}, \,\,\lambda\in\mathbb{R},t\in\Sigma=[\tau,1], \tau>0.
\end{split}\enq
For any $\tau\in(0,1]$ we recall, cf. \eqref{defOmt}, the rescaled trace map $\tr_t$ defined by the formula, 
\beq\lb{dfnts}
\tr_tu=(\gaD u, t^{-1}\gaN u),\, u\in\dom(\gaN)\subset H^1(\Om),\,\, t\in\Sigma=[\tau,1].
\enq
Then the following lemma holds, see \cite[Lemma 4.3 and Proposition 4.10]{CJLS14}.

\begin{lemma}\cite{CJLS14}\lb{GevL} Assume Hypothesis \ref{h1}. Let $L^{t}_{\cG}$ be the family of operators on $L^2(\Om;\bbR^N)$ defined in \eqref{L_t} for $t\in\Sigma=[\tau,1]$ and $\tau\in(0,1]$. Fix $\lambda\in\bbR$ and let $\Upsilon(\lambda,t)=\tr_t(\cK_{\lambda,t})$ be the family of the subspaces in $\cH=H^{1/2}(\dOm)\times H^{-1/2}(\dOm)$ defined by means of \eqref{dfnKS}, \eqref{dfnts} for $t\in\Sigma$. The path $t\to\Upsilon(\lambda,\cdot)$ is piece-wise $C^1(\Sigma; F\Lambda(\cG))$. Moreover, for any $\lambda_0\in\mathbb{R}$, $t_0\in\Sigma$ we have  
\begin{align}\lb{eq:4.23}
&t^2_0\lambda_0\in\Sp(L^{t_0}_{\cG}) \text{ if and only if $
\tr_{t_0}(\cK_{\lambda_0,t_0})\cap\cG\neq\{0\}$,  and } \\
& \dim_\bbR\big(\ker(L^{t_0}_{\cG}-t^2_0\lambda_0)\big)=\dim_\bbR\big(\tr_{t_0}(\cK_{\lambda_0,t_0})\cap\cG\big).\label{eqdims}\end{align} 
\end{lemma}

 %We note the following result found  in Theorem 2.6 and Corollary 2.7 of \cite{GM08}, see also Theorem \ref{th26GM} below.
%\begin{theorem}\lb{GMthm} Assume Hypothesis \ref{h1} (i), (iii), and consider a potential $V\in L^\infty(\Om)$. %\\Let $\Theta\in\cB(H^{1/2}(\dOm), H^{-1/2}(\dOm))$ be the selfadjoint operator associated by \eqref{qftheta} with the form $\aT$ and let $\cG=\gr(\Theta)$. Let $L_\cG\colon\dom(L_\cG)\subset L^2(\Om)\to L^2(\Om)$ denote the operator defined in $L^2(\Om)$ by
%\beq\lb{fdLG}L_\cG u=-\Delta u+Vu,\,
%\dom(L_\cG)=\big\{u\in H^1(\Om): \Delta u\in L^2(\Om), \tr u\in\cG\big\}.
%\enq
%Then Hypothesis \ref{h1}(iii) holds, that is,
%Then the operator $L_\cG^t$, $t\in[0,1]$, defined in \eqref{L_t} is selfadjoint and semibounded from below, has compact resolvent (and therefore only discrete spectrum) and satisfies $\dom\big(|L_\cG^t|^{1/2}\big)=H^1_0(\Om)$ for the Dirichlet case and 
%$\dom\big(|L_\cG^t|^{1/2}\big)=H^1(\Om)$ for the Neumann-based case.
%\end{theorem}

We mention an additional assumption on $\Theta$ under which one has the inclusion $\dom(L^{t}_{\cG})\subseteq H^2(\Omega)$, cf.\ Hypothesis \ref{h1} (iii'). Indeed, to ensure this,
 we must assume additional smoothness of the domain (as in Hypothesis \ref{h1}(i')) and impose a condition on the operator $\Theta$.
\begin{hypothesis} \lb{h3}
Assume Hypothesis \ref{h1}(i'), (iii). In addition, assume that the operator $\Theta$ satisfies\\ $\Theta\in\mathcal{B}_{\infty}(H^{3/2}(\dOm),H^{1/2}(\dOm))$. 
\end{hypothesis}
 In other words, if Hypothesis \ref{h3} holds, then
$$\Theta\in\mathcal{B}(H^{1/2}(\dOm),H^{-1/2}(\dOm))\cap\mathcal{B}_{\infty}(H^{3/2}(\dOm),H^{1/2}(\dOm))$$ and $\Theta^*=\Theta$ as an operator from $H^{1/2}(\dOm)$ into $H^{-1/2}(\dOm)$. 
A sufficient condition for the inclusion \\$\Theta\in\mathcal{B}_{\infty}(H^{3/2}(\dOm),H^{1/2}(\dOm))$ is $\Theta\in\mathcal{B}(H^{3/2-\varepsilon}(\dOm),H^{1/2}(\dOm))$ for some $\varepsilon>0$.
The following result can be found in \cite[Theorem 2.17]{GM08}.
\begin{theorem}\lb{GMthmH2}
Assume Hypothesis \ref{h3}. Then $\dom(L^{t}_{\cG})\subseteq H^2(\Om)$, that is, Hypothesis \ref{h1}(iii') holds.
\end{theorem}
Next, we recall a result from \cite{GM08} saying that the operator $L_\cG^t$ is selfadjoint and has compact resolvent
provided Hypotheses \ref{h1} (iii) hold.
%%%%%%%%%%%%%%%%%%%
Let $\cG$ be either the Dirichlet subspace or a Neumann-based subspace in the boundary space $\cH=H^{1/2}(\dOm)\times H^{-1/2}(\dOm)$ and consider the sesquilinear form $\mathfrak{l}_{\cG}^t$ on $L^2(\Om)$, defined for $t\in[\tau,1]$
by the formula
\begin{align}\lb{dfnlGt}
%\begin{split}
\mathfrak{l}_{\cG} ^t(u,v)=\langle\nabla u,\nabla v\rangle_{L^2(\Om)}+t^2\langle V(t x)u,v\rangle_{L^2(\Om)}-t \langle\Theta\gaD  u,\gaD   v\rangle_{1/2},\,
%\no\\
\dom(\mathfrak{l}_{\cG} ^t)=H_\cG^1(\Omega)\times H^1_\cG(\Omega),
%\end{split}
\end{align}
where $H^1_\cG(\Om)=H^1(\Om)$ provided $\cG$ is a Neumann-based subspace, and $H^1_\cG(\Om)=H_0^1(\Om)$ provided $\cG$ is the Dirichlet subspace (we let $\Theta=0$ for the Dirichlet case in \eqref{dfnlGt}).
One can associate to the form $\mathfrak{l}_{\cG} ^t$ the differential operator $L_{\cG}^t$ using the following result, which can be found in, e.g.,
\cite[Theorem 2.6]{GM08}. 
\begin{theorem}\lb{th26GM}
Assume Hypotheses \ref{h1} (i), (iii), and consider a potential $V\in L^\infty(\Om)$. Then the sesquilinear form $\mathfrak{l}_{\cG} ^t$ is symmetric, (uniformly) bounded from below and closed in 
$L^2(\Omega)\times L^2(\Omega)$. The associated 
operator $L_{\cG}^t$ in $L^2(\Om)$, satisfying the relation
\beq\lb{OPFORM}
\mathfrak{l}_{\cG} ^t(u,v)=\langle L_{\cG}^t u, v\rangle_{L^2(\Om)}\,\text{ for all }
\, u\in\dom(L_{\cG}^t), v\in H^1_{\cG}(\Om),
\enq is selfadjoint, bounded from below, has compact resolvent (and therefore only discrete spectrum), and is given by formulas \eqref{L_t}, that is,
\begin{align}
L_{\cG}^tu(x)&=-\Delta u(x)+t^2V(t x)u(x),\, x\in\Om, u\in\dom(L_{\cG}^t),\nonumber\\
 \dom(L^t_{\cG})&=\big\{u\in H^1(\Omega)\,|\,\Delta u\in L^2(\Omega) \,\text{ and }\\\nonumber
 &\gaD u=0 \,\,\hbox{in}\,\,H^{1/2}(\dOm)\,\,\text{if}\,\,\cG \,\text{is the Dirichlet subpace,}\,\,\,\text{or}\\ &(\gaN-t\Theta\gaD)u=0\,\,\hbox{in}\,\,H^{-1/2}(\dOm)\,\,\text{if}\,\,\cG \,\text{is the Neumann-based subpace}  \big\}.\nonumber
\end{align}
\end{theorem}

\subsection{Formulas for the crossing form}
We now recall general formulas for the Maslov crossing form associated with the path $t\mapsto\Upsilon(\lambda,t)=\tr_t(\cK_{\lambda,t})$; in general, the Maslov form is defined in \eqref{dfnMasForm}. Let $\cG$ be either the Dirichlet or a Neumann-based Lagrangian subspace of $\cH=H^{1/2}(\dOm)\times H^{-1/2}(\dOm)$. Let $\Sigma=[\tau,1]$. We begin with a version of \cite[Lemma 5.2]{CJLS14}
\begin{lemma}\label{prelmon} Assume Hypothesis \ref{h1} (i), (ii), (iii) and (iv'). Fix $\lambda_0\in\mathbb{R}$ and consider the path $t\mapsto\Upsilon(\lambda_0,t)=\tr_t(\cK_{\lambda_0,t})$. If $t_0\in\Sigma$ is a crossing and $q\in \tr_{t_0}(\cK_{\lambda_0,t_0})\cap\cG$, then there exists a unique function $u_{t_0} \in \cK_{\lambda_0,t_0}$ such that $q=\tr_{t_0} u_{t_0}$, and the Maslov crossing form satisfies
%there exists a segment $\Sigma_0\subset\Sigma$ containing $s_0$ such that for any $q\in T_{s_0}(\cK_{s_0})\cap\cG$ there exists a function $s\mapsto u_s$ in $C^1(\Sigma_0; H^1(\Om))$ such that $q=u_{s_0}$, $u_s\in\cK_s$ for each $s\in\Sigma_0$, and
\beq\label{mqq}
\mathfrak{m}_{t_0}(q,q)=\frac1{t_0}\langle \dot{V}_{\lambda_0,t}\big|_{t=t_0}
 u_{t_0}, u_{t_0}\rangle_{L^2(\Om)}-\frac{1}{t_0^{2}}\,\langle
\gaN u_{t_0},\gaD u_{t_0}\rangle_{1/2},
\enq
where ${V}_{\lambda_0,t}(x)=t^2V(tx)-t^2\lambda_0,\,\,x\in\Om$.
\end{lemma}

Next, we recall yet another general formula given in \cite[Lemma 5.5]{CJLS14} for the (Maslov) crossing form that holds for both the Dirichlet- and Neumann-based cases.
\begin{lemma}\lb{preltmon}  Assume Hypothesis \ref{h1} (i'), (ii), (iii') and  (iv'), and, for a fixed $\lambda_0$, let $t_0\in[\tau,1]$ be a crossing for the path $t\mapsto\Upsilon(\lambda_0,t)=\tr_t(\cK_{\lambda_0,t})$. %so that $T_{s_0}(\cK_{s_0})\cap\cG\neq\{0\}$. 
If $q\in \tr_{t_0}(\cK_{\lambda_0,t_0})\cap\cG$ and $u_{t_0}\in\dom\big(L^{t_0}_{\cG}\big)\subseteq H^2(\Om)$ are such that $(-\Delta +t_0^2V(t_0x)-t_0^2\lambda_0)u_{t_0}=0$ and $q=\tr_{t_0}u_{t_0}$ then the Maslov crossing form of $\Upsilon$ is given as follows:
\beq\lb{3henF1}\begin{split}
\mathfrak{m}_{t_0}(q,q)&=\frac1{t_0^2}\int_{\dOm}\Big(\big(\nabla u_{t_0}\cdot \nabla u_{t_0})(\nu\cdot x)-2\langle\nabla u_{t_0}\cdot x,\gaN^{\rm s}u_{t_0}(x)\rangle_{\bbR^N}\\&+(1-d)\langle\gaN^{\rm s} u_{t_0}(x), u_{t_0}(x)\rangle_{\bbR^N}+\langle (V(t_0x)-\lambda_0)u_{t_0}(x),u_{t_0}(x)\rangle_{\bbR^N}(\nu\cdot x)\Big)\,dx.
\end{split}\enq
\end{lemma} 
An eigenfunction $u_{t_0}$ of the operator  $L^{t_0}_{\cG}$ from \eqref{L_t} corresponding to its eigenvalue $t_0^2\lambda_0$ satisfies $u_{t_0}\in H^2(\Om)$ by the assumptions in the lemma and therefore the strong Neumann trace $\gaN^{\rm s}u_{t_0}$ is defined. 

Finally, the Maslov crossing form for the path $t\mapsto\Upsilon(\lambda_0,t)=\tr_t(\cK_{\lambda_0,t})$ when $t\in[\tau,1]$ and $\cG=\cH_D=\{(0,g): g\in H^{-1/2}(\dOm)\}$ is the Dirichlet subspace can be expressed as follows, see \cite[Corollary 5.7]{CJLS14}.

\begin{corollary}\lb{tmon} Assume Hypothesis \ref{h1} with (i'), $\cG=\cH_D$ and  (iv'). Fix $\lambda_0\in\mathbb{R}$ and consider the path $t\mapsto\Upsilon(\lambda_0,t)=\tr_t(\cK_{\lambda_0,t})$. Then any crossing $t_0\in[\tau,1]$ is negative. Specifically, if $q\in \tr_{t_0}(\cK_{\lambda_0,t_0})\cap\cH_D$ and $u_{t_0}\in\dom\big(L^{t_0}_{\cG}\big)=H^2(\Om)\cap H^1_0(\Om)$ are such that $(-\Delta +t_0^2V(t_0x)-t_0^2\lambda_0)u_{t_0}=0$ and $q=\tr_{t_0}u_{t_0}$, then the value of the Maslov crossing form is given as follows:
\beq\lb{henF}
\mathfrak{m}_{t_0}(q,q)=-\frac1{t_0^2}\int_{\dOm}\|\gaN^{\rm s} u_{t_0}\|_{\bbR^N}^2(\nu\cdot x)\,dx<0.
\enq
\end{corollary} 

%\bbox{Inserted}

Similar crossing form computations have appeared in \cite[Eqn.(5.32)]{DJ11}, \cite[Eqn.(32)]{PW} and \cite[Eqn.(35)]{CJM}. In fact, for $V=0$ the right-hand side of \eqref{henF} is the well-known Rayleigh--Hadamard formula for the derivative of a simple Dirichlet eigenvalue with respect to a domain perturbation; see e.g. \cite[Chapter 5]{Henry} and references therein.

%\bbox{End of the addition}

%%%%%%%%%%%%%%%%%%%%%%%%%%%%%%%%%%
\section{The proofs of the main results}\label{sec:mainres}

\subsection{Asymptotic expansions as $t\to t_0$}\lb{NC}
In this subsection we derive asymptotic formulas for the eigenvalues of the operator $L^t_{\cG}$ defined in \eqref{L_t}. Fix $t_0\in[0,1]$ and let $\lambda^{t_0}$ be an eigenvalue of $L^{t_0}_{\cG}$ of multiplicity $m$.  Our ultimate goal, cf.\ \eqref{eq1.84} and \eqref{eq1.85}, is an asymptotic formula for the eigenvalues $\lambda_j(t)$, $j=1,\dots,m$, of $L^{t}_{\cG}$ that bifurcate from the eigenvalue $\lambda^{t_0}$ for $t$ near $t_0$. First, we need some preliminaries. 
The following technical result can be found, e.g., in \cite[Lemma 2.5]{GM08}.
\begin{lemma}\label{var}
Assume that $\Omega$ is a nonempty open bounded Lipschitz domain in $\bbR^d$ for $d\ge2$. Then
for every $\varepsilon>0$ there exits a $\beta(\varepsilon)>0$ such that $\beta(\varepsilon)=O(1/\varepsilon)$ as $\varepsilon\to 0$ and 
\begin{equation}
\|\gaD u\|_{L^2(\dOm)}^2\le\varepsilon\|\nabla u\|_{L^2(\Om)}^2+\beta(\varepsilon)\| u\|_{L^2(\Om)}^2
\end{equation}
for any $u\in H^1(\Omega)$.
\end{lemma}
We now derive an estimate for the operator form $\mathfrak{l}^t_\cG$ defined in \eqref{OPFORM}.
\begin{lemma}\lb{lem:Gamma4} Assume Hypothesis \ref{h1} (i), (ii) and (iii). There exist positive constants  $\Lambda=\Lambda(\Theta )$ and $c(\Theta)$ such that \[(\mathfrak{l}_{\cG} ^{t}+\Lambda)(u,u)\geq c(\Theta)\|u\|_{H^1(\Om)}^2 \text{ for any $u\in\dom(\mathfrak{l}_{\cG} ^{t})$ and all $t\in[0,1]$}.\] %In particular, $\|(L^t_{\cG}+\Lambda)u\|_{L^2(\Om)}^2\geq c(\Theta)\|u\|_{H^1(\Om)}^2$ for any $u\in\dom(L^t_{\cG})$  . 
\end{lemma}
\begin{proof}
Let $u\in\dom(\mathfrak{l}_{\cG} ^{t})$. If $\cG=\cH_D$ then $\Lambda=\| V\|_{L^{\infty}(\Om)}+1$ and $c(\Theta)=1$ will do the job since $(\mathfrak{l}_{\cG} ^{t}+\Lambda)(u,u)\geq\|\nabla u\|_{L^2(\Om)}^2-\|V\|_{L^\infty(\Om)}\|u\|_{L^2(\Om)}^2+\Lambda\|u\|_{L^2(\Om)}^2$. If $\cG=\gr(\Theta)$ is a Neumann-based subspace then 
%\begin{align}
%\begin{split}
%\mathfrak{l}_{\cG} ^{t}(u,u)=\langle\nabla u,\nabla u\rangle_{L^2(\Om)}+&t^2\langle (V(t x))u,u\rangle_{L^2(\Om)}%\\&
%-t\langle\Theta\gaD  u,\gaD   u\rangle_{1/2}.
%\lb{dfnlGt6.2}
%\end{split}
%\end{align}
%By \eqref{propTheta2}, we infer
%\begin{equation}
%-t\langle \Theta \gaD u,\gaD u\rangle_{1/2}\ge-t c_\Theta\|\gaD u\|_{L^2(\dOm)}^2.
%\end{equation}
let us fix $\varepsilon\in(0,{1}/{c_\Theta})$. Then \eqref{propTheta2}, $t\le1$ and Lemma \ref{var} yield
\begin{equation}\lb{eq6.5}
\begin{split}
-t\langle \Theta \gaD u,\gaD u\rangle_{1/2}&\ge-c_\Theta\|\gaD u\|_{L^2(\dOm)}^2%\\&
\ge-c_\Theta\big(\varepsilon\|\nabla u\|_{L^2(\Om)}^2+\beta(\varepsilon)\| u\|_{L^2(\Om)}^2\big)
\end{split}
\end{equation}
for some $\beta(\varepsilon)>0$.
Choose $\Lambda\geq\| V\|_{L^{\infty}(\Om)}+(1+c_\Theta\beta(\varepsilon))$ and $c(\Theta)<1-c_\Theta\varepsilon$. Then for any $t\in[0,1]$ the estimate
$\langle (t^2V(tx)+\Lambda)u,u\rangle_{L^2}\ge \big(-\|V\|_{L^\infty(\Om)}\|u\|_{L^2(\Om)}^2+\Lambda\|u\|^2_{L^2(\Om)}\big)$
and equation \eqref{eq6.5}  yield
\begin{align}
\begin{split}
(\mathfrak{l}_{\cG} ^{t}+\Lambda)(u,u)&=\langle\nabla u,\nabla u\rangle_{L^2(\Om)}+\langle (t^2V(t x)+\Lambda)u,u\rangle_{L^2(\Om)}
%\\&\hskip3cm
 -t\langle\Theta\gaD  u,\gaD   u\rangle_{1/2}\\& \ge\|\nabla u\|_{L^2(\Om)}^2+(1+c_\Theta\beta(\varepsilon))\| u\|_{L^2(\Om)}^2-c_\Theta(\varepsilon\|\nabla u\|_{L^2(\Om)}^2+\beta(\varepsilon)\| u\|_{L^2(\Om)}^2)\\&\ge(1-c_\Theta\varepsilon)\|\nabla u\|_{L^2(\Om)}^2+\| u\|_{L^2(\Om)}^2\ge c(\Theta)\| u\|_{H^1(\Om)}^2.
\end{split}
\end{align}
\end{proof}
Following the general discussion of holomorphic families of closed unbounded operators in \cite[Section VII.1.2]{Kato}, we introduce our next definition.
\begin{definition}\label{cont}
A family of closed operators $\{T(t)\}_{t\in \Sigma}$ on a Hilbert space $\cX$ is said to be continuous on an interval $\Sigma\subset\bbR$ if there exists a Hilbert space $\cX'$ and continuous families of operators $\{U(t)\}_{t\in \Sigma}$ and $\{W(t)\}_{t\in \Sigma}$ in $\cB(\cX',\cX)$ such that $U(t)$ is a one-to-one map of $\cX'$ onto $\dom(T(t))$ and the identity $T(t)U(t)=W(t)$ holds for all $t\in \Sigma$.
\end{definition}
We recall that Hypothesis \ref{h1} (iv) yields
\begin{equation}\label{unif1}
\sup_{x\in\Om}\|V(t x)-V(t_0x)\|_{\bbR^{N\times N}}\to 0 \,\,\,\hbox{as}\,\,\, t\to t_0\,\text{ 
for any $t_0\in[0,1]$}.\end{equation}

\begin{lemma}\lb{4.5}
Assume Hypothesis \ref{h1} (i),  (iii) and (iv). Then the family $\{L_{\cG}^t\}_{t\in \Sigma}$ is continuous near $t_0$, that is, on some interval $\Sigma_{t_0}$ that contains $t_0$.
\end{lemma}
\begin{proof}
Letting $t=t_0$ in Lemma \ref{lem:Gamma4} yields
\[(\mathfrak{l}_{\cG} ^{t_0}+\Lambda)(u,u)\geq c(\Theta)\|u\|_{H^1(\Om)}\,\text{ for  all $u\in H^1_{\cG}(\Omega)$}.\] 
The selfadjoint operator associated with the form $\mathfrak{l}_{\cG} ^{t_0}+\Lambda$ by Theorem \ref{th26GM} is $L_{\cG}^{t_0}+\Lambda I_{L^2(\Om)}$. 
The operator $L_{\cG}^{t_0}+\Lambda I_{L^2(\Om)}$ is clearly invertible, and 
if $u\in H^1_{\cG}(\Omega)$, then \[(L_{\cG}^{t_0}+\Lambda I_{L^2(\Om)})^{-1/2}u\in\dom\big((L_{\cG}^{t_0}+\Lambda I_{L^2(\Om)})^{1/2}\big)=H^1_{\cG}(\Omega).\]
For the remainder the proof we let $G$ denote the operator
\[
	G = (L_{\cG}^{t_0}+\Lambda I_{L^2(\Om)})^{1/2}\colon H_{\cG}^{1}(\Omega) \rightarrow L^2(\Omega).
\]
We note, cf.\ (B.42) and (B.43) in \cite{GM08}, that
\beq\lb{6.6.1}
\|Gu\|^2_{L^2(\Omega)}=(\mathfrak{l}_{\cG} ^{t_0}+\Lambda )(u,u)\,\text{ for all }\, u\in H^1_{\cG}(\Om).\enq 
By Lemma \ref{lem:Gamma4}, $\|Gu\|^2_{L^2(\Omega)}\geq c(\Theta)\|u\|^2_{H_\cG^1(\Om)} $, hence $G^{-1} = (L_{\cG}^{t_0}+\Lambda I_{L^2(\Om)})^{-1/2}\in\cB(L^2(\Omega),H_{\cG}^{1}(\Omega))$.
%We can also view $G$ as an unbounded, selfadjoint operator on $L^2(\Om)$.
Now using \eqref{dfnlGt} we introduce a new sesquilinear form
\begin{align}
&\widetilde{\mathfrak{l}}_{\cG}^t(u,v)=\mathfrak{l}_{\cG} ^t\left(G^{-1}u,G^{-1}v \right),\,%\\&
\dom(\widetilde{\mathfrak{l}}_{\cG}^{t})=L^2(\Omega)\times L^2(\Omega).
\end{align}
It is easy to see that $\widetilde{\mathfrak{l}}_{\cG}^{t}$ is bounded on  $L^2(\Omega)\times L^2(\Omega)$ (cf. \cite[Lemma 6.3]{CJLS14}). Let $\widetilde{L}_{\cG}^t\in\cB(L^2(\Omega))$ be the selfadjoint operator associated with $\widetilde{\mathfrak{l}}_{\cG}^{t}$ by the First Representation Theorem \cite[Theorem VI.2.1]{Kato}. Then 
\begin{align}\lb{LLtil}
\langle \widetilde{L}_{\cG}^tu,v\rangle_{L^2(\Om)}=\widetilde{\mathfrak{l}}_{\cG}^{t}(u,v)=\mathfrak{l}_{\cG}^t\left(G^{-1}u,G^{-1}v\right) \,\text{ for all} \,u,v\in L^2(\Omega).\no
\end{align}
Next, following the proof of  \cite[Lemma 6.3]{CJLS14}, one can show that 
\begin{equation}\label{UV}
L_{\cG}^tU(t)=W(t)\,\text{ for $t$ near $t_0$},
\end{equation}
where $U(t)=(L_{\cG}^t+\Lambda I_{L^2(\Omega)})^{-1}$ and $W(t)=I_{L^2(\Om)}-\Lambda U(t)$ are the continuous operator families near $t_0$. Hence, according to Definition \ref{cont}, the family $\{L_{\cG}^t\}$ is continuous near $t_0$.
\end{proof}
We denote by
$R(\zeta,t)=\big(L_{\cG}^t-\zeta I_{L^2(\Om)}\big)^{-1}$, $\zeta\in\bbC\setminus\Sp(L_{\cG}^t)$, $t\in[0,1]$,
the resolvent operator for $L_{\cG}^t$ in $L^2(\Om)$.
\begin{lemma}\label{rescont}
Let $\zeta\in\bbC\setminus\Sp(L_{\cG}^{t_0})$. Then $\zeta\in\bbC\setminus\Sp(L_{\cG}^t)$ for $t$ near $t_0$. Moreover, the function 
$t \mapsto R(\zeta,t)\in\cB(L^2(\Om))$ is continuous for $t$ near $t_0$, uniformly for $\zeta$ in compact subsets of $\bbC\setminus\Sp(L_{\cG}^t)$.
\end{lemma}
\begin{proof}
Let $\zeta\in\bbC\setminus\Sp(L_{\cG}^{t_0})$. It follows from \eqref{UV} that
\begin{equation}\lb{eq1.26}
(L_{\cG}^t-\zeta)U(t)=W(t)-\zeta U(t)\, \text{ for $t$ near $t_0$}.
\end{equation}
The operator  $W(t_0)-\zeta U(t_0)=(L_{\cG}^{t_0}-\zeta)(L_{\cG}^{t_0}+\Lambda I_{L^2(\Om)})^{-1}$ is a bijection of $L^2(\Omega)$ onto $L^2(\Omega)$. Therefore, $W(t)-\zeta U(t)$
is boundedly  invertible for $t$ near $t_0$ since the function $t \mapsto W(t)-\zeta U(t)$ is continuous in $\cB(L^2(\Om))$, uniformly for $\zeta$ in compact subsets of $\bbC$. This implies that $\zeta\in\bbC\setminus\Sp(L_{\cG}^t)$ for $t$ near $t_0$, since using \eqref{eq1.26} it is easy to check that 
\begin{equation}
(L_{\cG}^t-\zeta)^{-1}=U(t)(W(t)-\zeta U(t))^{-1}\, \text{ for $t$ near $t_0$}.
\end{equation}
Hence, the function $t \mapsto R(\zeta,t)$ is continuous for $t$ near $t_0$ in the operator norm, uniformly in $\zeta$.
\end{proof}
\begin{remark}\label{rem4.6}  
For $\zeta\in\bbC\setminus\Sp (L_{\cG}^{t_0})$ one has
$[\gaD   R(\overline{\zeta}, t_0)]^*\in\cB(H^{-1/2}(\partial\Omega),H^{1}(\Omega))$ (see e.g., \cite[Equation  (4.7)]{GM08})  which, in particular, yields 
\begin{equation}\label{adj}
\langle [\gaD   R(\overline{\zeta}, t_0)]^*g,v)\rangle_{L^2(\Om)}=\langle g,\gaD   R(\overline{\zeta}, t_0)v\rangle_{1/2}
\end{equation}
for any $g\in H^{-1/2}(\partial\Omega)$ and $v\in L^2(\Omega)$. Moreover, the resolvent $R({\zeta}, t_0)$, $\zeta\in\bbC\setminus\Sp (L_{\cG}^{t_0})$, originally defined as a bounded operator on $L^2(\Omega)$ can be extended to a mapping in $\cB((H^{1}(\Omega))^*,H^{1}(\Omega))$ (see e.g., \cite[Theorem  4.5]{GM08}). The same assertion holds for a Riesz projection $P$ of the operator
$L_{\cG}^{t_0}$. In particular, one can rewrite the operator $[\gaD   R(\overline{\zeta}, t_0)]^*$ as
\begin{equation}\lb{comp}
[\gaD   R(\overline{\zeta}, t_0)]^*=R({\zeta}, t_0)\gaD^*\,\,\hbox{for}\,\, \zeta\in\bbC\setminus\Sp (L_{\cG}^{t_0}),
\end{equation}
where $\gaD^*\in\cB(H^{-1/2}(\partial\Omega),(H^{1}(\Omega))^*)$ and $R({\zeta}, t_0)\in\cB((H^{1}(\Omega))^*,H^{1}(\Omega))$ (see e.g., \cite[Corollary  4.7]{GM08}). Below, we do not distinguish between $R({\zeta}, t_0)$, respectively, $P$, as the bounded operator on $L^2(\Omega)$ and the extension of $R({\zeta}, t_0)$, respectively, $P$, to a mapping in $\cB((H^{1}(\Omega))^*,H^{1}(\Omega))$. \hfill$\Diamond$
\end{remark}

Our next lemma gives an asymptotic result for the difference of the resolvents of the operators $L_{\cG}^t$ and $L_{\cG}^{t_0}$ as $t\to t_0$. We recall the following notation:
\begin{equation}\label{T_D}
\Theta_D=\gaD^*\Theta\gaD\in\cB(H^{1}(\Omega),(H^{1}(\Omega))^*).
\end{equation}
\begin{lemma}\lb{lemma5.5}
If $\zeta\in\bbC\setminus\Sp(L_{\cG}^{t_0})$, then 
\begin{equation}\label{resolvent}
\begin{split}
R(\zeta, t)&-R(\zeta, t_0)\\&
=(t-t_0)R_1(\zeta, t_0)R(\zeta, t_0)+(t-t_0)^2\big(R_1^2(\zeta, t_0)R(\zeta, t_0)+R_2(\zeta, t_0)R(\zeta, t_0)\big)+r(t-t_0),
\end{split}
\end{equation}
where we introduce the notation
\begin{equation}\lb{4.19}
\begin{split}
R_1(\zeta, t_0)=-R(\zeta, t_0)\dot{ V}^t\big|_{t=t_0}+R({\zeta}, t_0)\Theta_D,\,
R_2(\zeta, t_0)= -\frac{1}{2}R(\zeta, t_0)\ddot{V}^t\big|_{t=t_0},
\end{split}
\end{equation} and $\|r(t)\|_{\cB(L^2(\Om))}=\mathrm{o}(t-t_0)^{2}$ as $t\to t_0$, uniformly for $\zeta$ in compact subsets of $\bbC\setminus\Sp(L_{\cG}^{t_0})$.
\end{lemma}
\begin{proof}
We recall that 
  $\zeta\in\bbC\setminus\Sp(L_{\cG}^t)$ for $t$ near $t_0$ by Lemma \ref{rescont}, since
  $\zeta\in\bbC\setminus\Sp(L_{\cG}^{t_0})$, and define $w=R(\zeta, t)u-R(\zeta, t_0)u$ for $\|u\|_{L^2(\Om)}\le1$. Using \eqref{OPFORM}, the definition of $w$ and the definition of the form $\mathfrak{l}_{\cG}^t(u,v)$, we obtain
\begin{align}\label{res}
\begin{split}
\langle w,v\rangle_{L^2(\Om)}&=({\mathfrak{l}}_{\cG}^{t_0}-\zeta)(R(\zeta, t_0)w,v)=({\mathfrak{l}}_{\cG}^{t_0}-\zeta)(w,R(\overline{\zeta}, t_0)v)\\
&=({\mathfrak{l}}_{\cG}^{t_0}-\zeta)(R(\zeta, t)u,R(\overline{\zeta}, t_0)v)-({\mathfrak{l}}_{\cG}^{t_0}-\zeta)(R(\zeta, t_0)u,R(\overline{\zeta}, t_0)v)\\
&=({\mathfrak{l}}_{\cG}^{t}-\zeta)(R(\zeta, t)u,R(\overline{\zeta}, t_0)v)-\langle (V^t(x)-V^{t_0}(x))R(\zeta, t)u,R(\overline{\zeta}, t_0)v\rangle_{L^2(\Om)}\\&
 \quad+(t-t_0) \langle\Theta\gaD  R(\zeta, t)u,\gaD   R(\overline{\zeta}, t_0)v\rangle_{1/2}-({\mathfrak{l}}_{\cG}^{t_0}-\zeta)(R(\zeta, t_0)u,R(\overline{\zeta}, t_0)v)\\
&=\langle u,R(\overline{\zeta}, t_0)v\rangle_{L^2(\Om)}-\langle (V^t(x)-V^{t_0}(x))R(\zeta, t)u,R(\overline{\zeta}, t_0)v\rangle_{L^2(\Om)}\\&
 \quad+(t-t_0) \langle\Theta\gaD  R(\zeta, t)u,\gaD   R(\overline{\zeta}, t_0)v\rangle_{1/2}-\langle u,R(\overline{\zeta}, t_0)v\rangle_{L^2(\Om)}\\
 &=-\langle (V^t(x)-V^{t_0}(x))R(\zeta, t)u,R(\overline{\zeta}, t_0)v\rangle_{L^2(\Om)}\\&
  \hskip3cm+(t-t_0) \langle\Theta\gaD  R(\zeta, t)u,\gaD   R(\overline{\zeta}, t_0)v\rangle_{1/2}.
\end{split}
\end{align}
Using \eqref{res} and \eqref{adj} we arrive at the formula
\begin{align}\lb{4.22}
\begin{split}
&\langle w,v\rangle_{L^2(\Om)}=-\langle R(\zeta, t_0)(V^t(x)-V^{t_0}(x))R(\zeta, t)u,v\rangle_{L^2(\Om)}\\&
  \hskip3cm+(t-t_0) \langle [\gaD   R(\overline{\zeta}, t_0)]^*\Theta\gaD  R(\zeta, t)u,v\rangle_{L^2(\Om)}.
\end{split}
\end{align}
Using \eqref{4.22} and \eqref{T_D}, we infer
\begin{align}\label{formula}
\begin{split}
&R(\zeta, t)=R(\zeta, t_0)-R(\zeta, t_0)(V^t(x)-V^{t_0}(x))R(\zeta, t)
 %\\&\hskip3cm 
 +(t-t_0)R({\zeta}, t_0)\Theta_D  R(\zeta, t).
\end{split}
\end{align}
Next, we decompose $V^t(x)$ as $t\to t_0$ up to quadratic terms:
\begin{equation}\lb{Taylor}
V^t(x)=V^{t_0}(x)+\dot{V}^t\big|_{t=t_0}(t-t_0)+\frac{1}{2}\ddot{V}^t\big|_{t=t_0}(t-t_0)^2+\mathrm{o}(t-t_0)^2.
\end{equation}
Replacing $V^t(x)$ in the right-hand  side of \eqref{formula} by \eqref{Taylor} yields
\begin{align}\lb{Ru}
\begin{split}
R(\zeta, t)&=R(\zeta, t_0)-(t-t_0)R(\zeta, t_0)\dot{ V}^t\big|_{t=t_0}R(\zeta, t)\\&+(t-t_0)R({\zeta}, t_0)\Theta_D  R(\zeta, t)-(t-t_0)^2R(\zeta, t_0)\frac{1}{2}\ddot{V}^t\big|_{t=t_0}R(\zeta, t)-R(\zeta, t_0)\mathrm{o}(t-t_0)^2R(\zeta, t),\\
&=R(\zeta, t_0)+(t-t_0)R_1(\zeta, t_0)R(\zeta, t)+(t-t_0)^2R_2(\zeta, t_0)R(\zeta, t)-R(\zeta, t_0)\mathrm{o}(t-t_0)^2R(\zeta, t),\\
\end{split}
\end{align}
where $R_1$ and $R_2$ are given in \eqref{4.19}.
%\begin{equation}
%\begin{split}
%R_1(\zeta, t_0)=&-R(\zeta, t_0)\dot{ V}^t\big|_{t=t_0}+R({\zeta}, t_0)\Theta_D,\\
%R_2(\zeta, t_0)=& -\frac{1}{2}R(\zeta, t_0)\ddot{V}^t\big|_{t=t_0}.
%\end{split}
%\end{equation}
Replacing $ R(\zeta, t)$ in the right-hand  side of \eqref{Ru} by \eqref{Ru} again yields
\begin{align*}
\begin{split}
R(\zeta, t)&=R(\zeta, t_0)+(t-t_0)R_1(\zeta, t_0)\Big(R(\zeta, t_0)+(t-t_0)R_1(\zeta, t_0)R(\zeta, t)+(t-t_0)^2R_2(\zeta, t_0)R(\zeta, t)\\
&\quad-R(\zeta, t_0)\mathrm{o}(t-t_0)^2R(\zeta, t)\Big)+(t-t_0)^2R_2(\zeta, t_0)\Big(R(\zeta, t_0)+(t-t_0)R_1(\zeta, t_0)R(\zeta, t)\\&\quad+(t-t_0)^2R_2(\zeta, t_0)R(\zeta, t)-R(\zeta, t_0)\mathrm{o}(t-t_0)^2R(\zeta, t)\Big)-R(\zeta, t_0)\mathrm{o}(t-t_0)^2R(\zeta, t)\\
&=R(\zeta, t_0)+(t-t_0)R_1(\zeta, t_0)R(\zeta, t_0)+(t-t_0)^2(R_1^2(\zeta, t_0)R(\zeta, t_0)+R_2(\zeta, t_0)R(\zeta, t_0))+r(t-t_0).
\end{split}
\end{align*}
Finally, we remark that $\|R(\zeta, t)-R(\zeta, t_0)\|_{\cB(L^2(\Om))}\to0$ and $\|R(\zeta,t)\|_{\cB(L^2(\Om))}$ is bounded as $t\to t_0$ by Lemma \ref{rescont} and thus, using \eqref{unif1}, we conclude that $\|r(t)\|_{\cB(L^2(\Om))}=\mathrm{o}(t-t_0)^{2}$ as $t\to t_0$, uniformly for $\zeta$ in compact subsets of $\bbC\setminus\Sp(L^{t_0}_{\cG})$.
\end{proof}

Our next objective is to study the asymptotics of the eigenvalues of $L_{\cG}^t$ that bifurcate from the eigenvalue $\lambda^{t_0}$ of $L_{\cG}^{t_0}$ for $t$ near $t_0$.
We let $P$ denote the orthogonal Riesz projection for $L_{\cG}^{t_0}$ corresponding to the eigenvalue $\lambda^{t_0}\in\Sp(L_{\cG}^{t_0})$ and let $\ran({P})=\ker(L_{\cG}^{t_0}-\lambda^{t_0}I_{L^2(\Om)})$ be the $m$-dimensional subspace in $L^2(\Om; \bbC^N)$ spanned by the normalized vector valued functions $\{u^{t_0}_j\}_{j=1}^m$.
  Also, we
let $P(t)$ denote the Riesz spectral protection for $L_{\cG}^t$ corresponding to the eigenvalues $\{\lambda_j^t\}_{j=1}^m\subset\Sp(L_{\cG}^t)$ bifurcating from $\lambda^{t_0}$, that is, we let
\beq\lb{dfnRPr}
P=\frac{1}{2\pi i}\int_\gamma(\zeta-L_{\cG}^{t_0})^{-1}\,d\zeta,\,
P(t)=\frac{1}{2\pi i}\int_\gamma(\zeta-L_{\cG}^t)^{-1}\,d\zeta.
\enq
Here, we first choose $\gamma$ so small that $\lambda^{t_0}$ is the only element in $\Sp(L^{t_0}_{\cG})$ surrounded by $\gamma$. Next, we choose $t$ close to $t_0$ so that the total multiplicity of all eigenvalues of $L^{t}_{\cG}$ enclosed by $\gamma$ is equal to the multiplicity of $\lambda^{t_0}$. This is indeed possible since $L^{t}_{\cG}$ is a continuous family by Lemma \ref{4.5}. 
Our objective is to establish an asymptotic formula for the eigenvalues $\lambda_j^t$ as $t\to t_0$ similar to \cite[Theorem II.5.11]{Kato}, which is valid for families of bounded operators on finite dimensional spaces. We stress that one can not directly use a related result \cite[Theorem VIII.2.9]{Kato} for families of unbounded operators, as the $t$-dependence of $L_{\cG}^t$ in our case  is more complicated than allowed in the latter theorem. We are thus forced to mimic the main strategy of \cite{Kato} in order to extend the relevant results to the family of the operators $L_{\cG}^t$.

We recall that $(L_{\cG}^{t_0}-\lambda^{t_0}I_{L^2(\Om)})P=0$ and thus the standard formula for the resolvent decomposition, see, e.g., \cite[Section III.6.5]{Kato}, yields
%\begin{align}\lb{decomp}
%R(\zeta,0)=&(-\zeta)^{-1}P+\sum_{n=0}^\infty\zeta^nS^n,\\&\text{
%where $S=\frac{1}{2\pi i}\int_\gamma\zeta^{-1}R(\zeta,0)\,d\zeta$}\lb{decomp1}\end{align}
\begin{equation}\lb{decomp}
	R(\zeta,t_0)=(\lambda^{t_0}-\zeta)^{-1}P+\sum_{n=0}^\infty(\zeta-\lambda^{t_0})^nS^{n+1},
\end{equation}	
where
\begin{equation}
	S=\frac{1}{2\pi i}\int_\gamma(\zeta-\lambda^{t_0})^{-1}R(\zeta,t_0)\,d\zeta \lb{decomp1}
\end{equation}
is the reduced resolvent for the operator $L_{\cG}^{t_0}$ in $L^2(\Om)$; in particular, $PS=SP=0$.

We introduce the notation $D(t)=P(t)-P$. Applying $-\frac{1}{2\pi i}\int_\gamma(\cdot)\,d\zeta$ in \eqref{resolvent} and using \eqref{dfnRPr} we conclude that $D(t)=
\mathrm{o}(t-t_0)$ as $t\to t_0$. Therefore $I_{L^2(\Om)}-D^2(t)$ is strictly positive for $t$ near $t_0$ and, following \cite[Section I.4.6]{Kato}, we may introduce mutually inverse operators $U(t)$ and $U(t)^{-1}$ in $\cB(L^2(\Om))$ as follows:
\begin{equation}\label{dfnUUinv}
\begin{split}
U(t)&=(I-D^2(t))^{-1/2}\big((I-P(t))(I-P)+P(t)P\big),\\
U(t)^{-1}&=(I-D^2(t))^{-1/2}\big((I-P)(I-P(t))+PP(t)\big).
\end{split}
\end{equation}
The transformation operator $U(t)$ splits the projections $P$ and $P(t)$, that is, 
\begin{equation}\label{up}
U(t)P=P(t)U(t),
\end{equation}
so that $U(t)$ is an isomorphism of the $m$-dimensional subspace $\ran({P})$ onto the subspace $\ran({P(t)})$. Using Lemma \ref{lemma5.5} and recalling notations \eqref{T_D}, \eqref{decomp1} we obtain the following result.
\begin{lemma}\lb{lem:simile} Let $P$ be the Riesz projection for $L_{\cG}^{t_0}$ onto the subspace $\ran({P})=\ker(L_{\cG}^{t_0}-\lambda^{t_0}I_{L^2(\Om)})$, let $P(t)$ be the respective Riesz projection for $L_{\cG}^t$ from \eqref{dfnRPr}, and  let the transformation operators $U(t)$ and $U(t)^{-1}$ be defined in \eqref{dfnUUinv}. Then, for $t$ near $t_0$, 
the operator $L^t_\cG\big|_{\ran({P(t)})}$ is similar to the finite dimensional operator  $T(t)=PU(t)^{-1}L_{\cG}^tP(t)U(t)P$ acting in the subspace $\ran({P})$ and satisfying the following asymptotic formula:
\begin{align}\label{pulup}
\begin{split}
PU(t)^{-1}L_{\cG}^tP(t)U(t)P
&=\lambda^{t_0}P+(t-t_0)P\big(\dot{V}^t\big|_{t=t_0}-\Theta_D  \big)P\\&+(t-t_0)^2P\big(\frac{1}{2}\ddot{V}^t\big|_{t=t_0}-\dot{V}^t\big|_{t=t_0}S\dot{V}^t\big|_{t=t_0}- \Theta_D   S\Theta_D  \\
&\qquad+\dot{V}^t\big|_{t=t_0} S\Theta_D  + \Theta_D S\dot{V}^t\big|_{t=t_0}\big)P+\mathrm{o}(t-t_0)^2
\text{ as $t\to t_0$}.
\end{split}
\end{align}
\end{lemma}
The proof of the lemma is a straightforward but lengthy calculation and is given in the next subsection. 

We will now proceed with the proof of the main results of the paper for the first derivative of the eigenvalues of the operators $L^{(\Om_t)}_\cG$ and $L^t_\cG$. By Lemma \ref{lem:simile} the spectrum of the operator $L^t_\cG$ is the same as the spectrum of the operator $T(t)$. We will use an asymptotic formula from \cite{Kato} for the eigenvalues of the operator $T(t)=
PU(t)^{-1}L_{\cG}^tP(t)U(t)P$ acting in the finite dimensional space $\ran(P)$, where
\begin{equation}\lb{T(t)}
	T(t)=\lambda^{t_0}T+(t-t_0) T^{(1)}+(t-t_0)^2T^{(2)}+\mathrm{o}(t-t_0)^2
	\text{ as $t\to t_0$},
\end{equation}
and we denote
\begin{align}\lb{4T1}
T&=I,\quad T^{(1)}=P(\dot{V}^t\big|_{t=t_0}-\Theta_D  )P,\quad\\\lb{T2} T^{(2)}&=P(\frac{1}{2}\ddot{V}^t\big|_{t=t_0}-\dot{V}^t\big|_{t=t_0}S\dot{V}^t\big|_{t=t_0}- \Theta_D   S\Theta_D  
+\dot{V}^t\big|_{t=t_0} S\Theta_D  + \Theta_D S\dot{V}^t\big|_{t=t_0})P.\end{align}
Indeed, let $\{\lambda^{(1)}_j\}_{j=1}^{m}$ denote the eigenvalues of the 
operator $T^{(1)}$ (some of them could be repeated). According to \cite[Theorem II.5.11]{Kato} the eigenvalues $\lambda_{j}^t$ of the operator $T(t)$ are given by the formula
\beq\lb{eq1.84}
\lambda_{j}^t=\lambda^{t_0}+(t-t_0)  \lambda^{(1)}_j+\mathrm{o}(t-t_0)
\,\text{ as $t\to t_0$}, \, j=1,\dots,m.
\enq
We stress again that $\{\lambda^{t}_j\}_{j=1}^{m}$, the eigenvalues of the operator $T(t)$, form the part of the spectrum $\Sp(L^t_{\cG})$ enclosed by $\gamma$ for $t$ near $t_0$.

\begin{proof}[Proof of Theorem \ref{gH}]
 By Lemma \ref{lemResc}, we have the relation $\lambda_j^t=t^2\lambda_j(t)$ for  $\lambda_j^t\in\Sp(L^{t}_{\cG})$ and $\lambda_j(t)\in\Sp(L_{\cG}^{\Om_t})$. We differentiate this relation with respect to $t$ and obtain
\begin{equation}\lb{4.54}
\dot\lambda_j(t)=\frac{1}{t^2}(\dot\lambda_j^t-2t\lambda_j(t)).
\end{equation}
 From formula \eqref{eq1.84}, we know that $\dot\lambda_j^{t}\big|_{t=t_0}=\lambda^{(1)}_j$ are the eigenvalues of the operator $T^{(1)}$ introduced in \eqref{4T1}. This proves the first equality in \eqref{1.17}. For the operator $T^{(1)}$ given in \eqref{4T1}, we now consider the associated quadratic form \[\mathfrak{t}^{(1)}(u,v)=(\dot{ V}^t\big|_{t=t_0}u,v)_{L^2(\Omega)}-\langle\Theta\gaD  u,\gaD   v\rangle_{1/2}, \, \dom(\mathfrak{t}^{(1)})=\ran(P)\times\ran({P}).\] Here, $u\in\dom (L_{\cG}^{t_0})$ since $u\in\ran(P)$. In particular, $\gaN u-t_0\Theta\gaD  u=0$. We rewrite $\mathfrak{t}^{(1)}$ as follows,
\begin{align}\lb{t1}
\begin{split}
\mathfrak{t}^{(1)}(u,v)=\langle\dot{V}^t\big|_{t=t_0}u,v\rangle_{L^2(\Omega)}-\frac{1}{t_0}\langle\gaN  u,\gaD   v\rangle_{1/2},\,\,\dom (\mathfrak{t}^{(1)})=\ran(P)\times\ran(P).
\end{split}
\end{align}
Next, for the normalized basic vectors $u_j^{t_0}\in\ran(P)$ we have $\lambda^{(1)}_j=\mathfrak{t}^{(1)}(u_j^{t_0},u_j^{t_0})$, $\, j=1,\dots,m$. Furthermore, for $q_j=(\gaD u_j^{t_0}, \frac{1}{t_0}\gaN u_j^{t_0})\in\Upsilon(\lambda(t_0), t_0)\cap\cG$  Lemma \ref{prelmon} yields
\beq\lb{eqn4.56}
\mathfrak{m}_{t_0}(q_j,q_j)=\frac1{t_0}\langle \dot{V}_{\lambda(t_0),t}\big|_{t=t_0}
 u_j^{t_0}, u_j^{t_0}\rangle_{L^2(\Om)}-\frac{1}{(t_0)^{2}}\,\langle
\gaN u_j^{t_0},\gaD u_j^{t_0}\rangle_{1/2}.
\enq
This gives formula \eqref{1.18}.
Using \eqref{t1}, relation \eqref{eq2.13} between $V^t$ and $V_{\lambda(t_0),t}$, and \eqref{eqn4.56}, we conclude that
\begin{equation}\label{eqn4.57}
\mathfrak{m}_{t_0}(q_j,q_j)=\frac1{t_0}\big(\mathfrak{t}^{(1)}(u_j^{t_0},u_j^{t_0})-\langle 2t_0\lambda(t_0)u_j^{t_0},u_j^{t_0}\rangle_{L^2(\Om)}\big)=\frac{1}{t_0}\big(\lambda_j^{(1)}-2t_0\lambda(t_0)\big).
\end{equation}
%Then
%\begin{equation}
%\frac1{t_0}\mathfrak{m}_{t_0}(q_j,q_j)=\frac1{(t_0)^2}(t^{(1)}(u_j^{t_0},u_j^{t_0})-\langle 2t_0\lambda(t_0)u_j^{t_0},u_j^{t_0}\rangle_{L^2(\Om)})=\frac1{(t_0)^2}(\lambda^{(1)}_j-2t_0\lambda(t_0)).
%\end{equation}
%Since, by Lemma \ref{prelmon}, there is a one-to-one correspondence between the $q$ functions and eigenfunctions $u_{t}$. 
%Therefore by \eqref{4.54}, the derivatives $\frac{d\lambda_j(t)}{dt}$ of $\lambda_j(t)$ at $t_0$ are computed as follows:
%\begin{equation}
%\frac{d\lambda_j(t)}{dt}\Big|_{t=t_0}=\frac1{t_0}\mathfrak{m}_{t_0}(q_j,q_j).
%\end{equation}
This proves the second equality in \eqref{1.17} and concludes the proof of the theorem.
\end{proof}
\begin{remark} Formula \eqref{eqn4.57} shows that if the form
$\widetilde{\mathfrak{t}}^{(1)}(u,v)=\frac{1}{t_0}\mathfrak{t}^{(1)}(u,v)-2t\lambda(t_0)\langle u,v\rangle_{L^2(\Omega)}$
on $\ran({P})$ is nondegenerate then the signature of the crossing (Maslov) form $\mathfrak{m}_{t_0}$ can be computed as follows:
\begin{align*}
\sgn(\mathfrak{m}_{t_0})&=\sgn(\widetilde{\mathfrak{t}}^{(1)})=
\card\big\{j: \lambda_j^{(1)}> 2t_0\lambda(t_0)\big\}-\card\big\{j: \lambda_j^{(1)}< 2t_0\lambda(t_0)\big\}\\
&=\card\big\{j: \dot{\lambda}_j(t)> 0\big\}-\card\big\{j: \dot{\lambda}_j(t)< 0\big\},
\end{align*}
and therefore the Maslov index of the path $t\mapsto\Upsilon(\lambda(t_0),t)$ near $t_0$ is equal to $\sgn(\widetilde{\mathfrak{t}}^{(1)})$.
\hfill$\Diamond$\end{remark}

\begin{proof}[Proof of Theorem \ref{1.3}]
The assertion immediately follows from Theorem \ref{gH} and Lemma \ref{preltmon}.
\end{proof}
\begin{proof}[Proof of Corollary \ref{1.4}]
This follows from Theorem \ref{gH} and Lemma \ref{tmon}.
\end{proof}

We will now derive formulas for the second derivatives of the eigenvalues of the operators $L^{(\Om_t)}_\cG$ and $L^t_{\cG}$. To this end, we will renumber the eigenvalues of the operator $T^{(1)}$ from \eqref{4T1}, and let $\{\lambda^{(1)}_i\}_{i=1}^{m'}=\{\lambda_j^{(1)}\}_{j=1}^m$ denote the 
$m'$ {\em distinct} eigenvalues of the operator $T^{(1)}$, so that $m'\le m$. We 
let $m_i^{(1)}$ denote the multiplicity of $\lambda^{(1)}_i$, and let $P^{(1)}_i$ 
denote the respective orthogonal Riesz spectral projections of $T^{(1)}$ corresponding to the eigenvalue $\lambda_i^{(1)}$ so that $m^{(1)}_i=\dim\ran(P^{(1)}_i)$, $\sum_{i=1}^{m'}P^{(1)}_i=P$ and  $\sum_{i=1}^{m'}m_i^{(1)}=m$.
Next, let $\lambda^{(2)}_{i k}$, $i=1,\dots, m'$, $k=1,\dots,m^{(1)}_i$,  denote the eigenvalues of the operator $P^{(1)}_i T^{(2)} P^{(1)}_i$ 
in $\ran(P^{(1)}_i)$ where the operator $T^{(2)}$ is defined in \eqref{T2}.

\begin{lemma}\label{l4.10}
The eigenvalues $\{\lambda^{t}_{ik}\}=\{\lambda_j^t\}_{j=1}^m$ of the operator $T(t)$ or $L^t_{\cG}$ have the asymptotic expansions
\beq\lb{eq1.85}
\lambda^t_{ik}=\lambda^{t_0}+(t-t_0)  \lambda^{(1)}_i+(t-t_0)^2\lambda^{(2)}_{ik}+\mathrm{o}(t-t_0)^2
\,\text{ as $t\to t_0$}, \, i=1,\dots,m',\, k=1,\dots,m^{(1)}_i.
\enq
\end{lemma}
\begin{proof}
This follows from \eqref{T(t)} and \cite[Theorem II.5.11]{Kato}.
\end{proof}
\begin{proof}[Proof of Theorem \ref{gHaf}] Lemma \ref{l4.10} yields $\dot{\lambda}^t_{ik}\big|_{t=t_0}=\lambda_i^{(1)}$ and $\frac12\ddot{\lambda}_{ik}^t\big|_{t=t_0}=\lambda_{ik}^{(2)}$. Differentiating formula ${\lambda}_{ik}^t=t^2\lambda_{ik}(t)$ at $t=t_0$ and solving for $\dot\lambda_{ik}(t_0)$ and $\ddot\lambda_{ik}(t_0)$ yields \eqref{asform}.
\end{proof}

\subsection{The proof of Lemma \ref{lem:simile}}
 We will split the proof into several steps.

{\em Step 1.}\, We first prove the following asymptotic relations for $\zeta\in\gamma$:
\begin{align}
\label{rp}
R(\zeta, t)P&=(\lambda^{t_0}-\zeta)^{-1}P+(t-t_0) (\lambda^{t_0}-\zeta)^{-1}(-R(\zeta, t_0)\dot{ V}^t\big|_{t=t_0}+R({\zeta}, t_0)\Theta_D) P+\mathrm{o}(t-t_0)_u,\\
PR(\zeta, t)
  &=(\lambda^{t_0}-\zeta)^{-1}P+(t-t_0) (\lambda^{t_0}-\zeta)^{-1}(-P\dot{ V}^t\big|_{t=t_0}+ P\Theta_D) R(\zeta, t_0)+\mathrm{o}(t-t_0)_u,\label{eq1.63}\\
PR(\zeta, t)P&=(\lambda^{t_0}-\zeta)^{-1}P+(t-t_0) (\lambda^{t_0}-\zeta)^{-2}(-P\dot{ V}^t\big|_{t=t_0}+P\Theta_D )P\nonumber\\&\quad+(t-t_0)^2(\lambda^{t_0}-\zeta)^{-2}(\tilde{R}_2(\zeta, t_0)-\frac{1}{2}P\ddot{V}^t\big|_{t=t_0}P)+\mathrm{o}(t-t_0)^{2}_u, \label{prp}\\
(I-P)R(\zeta, t)P&=(t-t_0)  (\lambda^{t_0}-\zeta)^{-1}(I-P)(-R(\zeta, t_0)\dot{ V}^t\big|_{t=t_0}+R({\zeta}, t_0)\Theta_D) P+\mathrm{o}(t-t_0)_u,\label{7.44.2}
\end{align}
where 
\begin{align*}
\tilde{R}_2(\zeta, t_0)&=PR_1^2(\zeta, t_0)R(\zeta, t_0)P=P\dot{V}^t\big|_{t=t_0}R(\zeta, t_0)\dot{V}^t\big|_{t=t_0}P+ P\Theta_D R(\zeta, t_0)  \Theta_D  P\\
&-P\dot{V}^t\big|_{t=t_0} R(\zeta, t_0)\Theta_D  P- P\Theta_D R(\zeta, t_0)\dot{V}^t\big|_{t=t_0}P.
\end{align*}
Here and below we write $\mathrm{o}(t-t_0)^\alpha_u$ to indicate a term which is $\mathrm{o}((t-t_0)^\alpha)$ as $t\to t_0$ uniformly for $\zeta\in\gamma$.
Indeed, to prove \eqref{rp} we note that $R(\zeta,t_0)P=(\lambda^{t_0}-\zeta)^{-1}P$ by \eqref{decomp} and use \eqref{resolvent}. Similarly, \eqref{resolvent} yields \eqref{eq1.63} and
\eqref{prp}. Also, \eqref{7.44.2} follows immediately from \eqref{rp}.

{\em Step 2.}\, We claim the following asymptotic relations for the Riesz projections:
\begin{align}\label{ptp}
P(t)P&=P+(t-t_0)(-S\dot{ V}^t\big|_{t=t_0}+ S\Theta_D) P+\mathrm{o}(t-t_0),\\
\label{ppt}
PP(t)&=P+(t-t_0)(-P\dot{ V}^t\big|_{t=t_0}+ P\Theta_D) S+\mathrm{o}(t-t_0),\\
\label{ppp}
PP(t)P&=P-(t-t_0)^2(P\dot{V}^t\big|_{t=t_0}S^2\dot{V}^t\big|_{t=t_0}P+ P\Theta_D   S^2\Theta_D  P\nonumber\\
&-P\dot{V}^t\big|_{t=t_0} S^2\Theta_D  P- P\Theta_D S^2\dot{V}^t\big|_{t=t_0}P)+\mathrm{o}(t-t_0)^{2}.
\end{align}
This follows from \eqref{dfnRPr} and \eqref{decomp1} by applying integration  $-\frac{1}{2\pi i}\int_\gamma(\cdot)\,d\zeta$ in \eqref{rp}, \eqref{eq1.63} and \eqref{prp}.

{\em Step 3.}\, We next claim the following asymptotic relations for the transformation operators defined in 
\eqref{dfnUUinv}:
\begin{align}
U(t)&=I+(t-t_0) ((-S\dot{ V}^t\big|_{t=t_0}+ S\Theta_D) P-(-P\dot{ V}^t\big|_{t=t_0}+ P\Theta_D) S)+\mathrm{o}(t-t_0),\lb{eq1.72}\\
U(t)^{-1}&=I+(t-t_0) ((-P\dot{ V}^t\big|_{t=t_0}+ P\Theta_D) S-(-S\dot{ V}^t\big|_{t=t_0}+ S\Theta_D) P)+\mathrm{o}(t-t_0),\label{ut-1}\\
PU(t)P&=P-\frac{1}{2}(t-t_0)^2 (P\dot{ V}^t\big|_{t=t_0}S^2\dot{ V}^t\big|_{t=t_0}P-P\dot{ V}^t\big|_{t=t_0}S^2\Theta_D P-P\Theta_D S^2\dot{ V}^t\big|_{t=t_0}P\nonumber\\
&+ P\Theta_D S^2\Theta_D P)+\mathrm{o}(t-t_0)^2,\label{pup}\\
\label{u-1}
PU(t)^{-1}P&=P-\frac{1}{2}(t-t_0)^2 (P\dot{ V}^t\big|_{t=t_0}S^2\dot{ V}^t\big|_{t=t_0}P-P\dot{ V}^t\big|_{t=t_0}S^2\Theta_D P-P\Theta_D S^2\dot{ V}^t\big|_{t=t_0}P\nonumber\\
&+ P\Theta_D S^2\Theta_D P)+\mathrm{o}(t-t_0)^2,\\
PU(t)^{-1}(I-P)&=(t-t_0) (-P\dot{ V}^t\big|_{t=t_0}S+P\Theta_D S)+\mathrm{o}(t-t_0).\label{7.44.1}
\end{align}
Indeed, recalling that $D(t) = P(t) - P$ and using \eqref{ptp} and \eqref{ppt} yields
\begin{align}\label{qq}
D^2(t)&=(P(t)-P)(P(t)-P)=P(t)+P-P(t)P-PP(t)\\
&=P(t)-P+(P-P(t)P)+(P-PP(t))=D(t)-(t-t_0) P^{(1)}+\mathrm{o}(t-t_0),
\no
\end{align}
where we have defined $P^{(1)}=(-S\dot{ V}^t\big|_{t=t_0}+S\Theta_D) P+(-P\dot{ V}^t\big|_{t=t_0}+ P\Theta_D) S$. Hence,
\begin{align}\label{qqq}
(I-D(t))(D(t)-(t-t_0) P^{(1)})=(t-t_0)  D(t)P^{(1)}+\mathrm{o}(t-t_0)=\mathrm{o}(t-t_0),
\end{align}
and therefore $D(t)=(t-t_0)  P^{(1)}+(I-D(t))^{-1}\mathrm{o}(t-t_0)$, yielding
\begin{align}\label{qqqq}
D(t)=(t-t_0)  P^{(1)}+\mathrm{o}(t-t_0).
\end{align}
Since $D(t)=O(t-t_0)$, formula \eqref{dfnUUinv} yields
$U(t)=I-P-P(t)+2P(t)P+O(t-t_0)^2$.
Applying \eqref{ptp} and then \eqref{qqqq}, we obtain
\eqref{eq1.72}. Similarly,
$U(t)^{-1}=I-P-P(t)+2PP(t)+O(t-t_0)^2$ yields \eqref{ut-1}.
Formula \eqref{pup} follows from the calculation
\begin{equation*}
\begin{split}
PU(t)P&=P(I-D^2(t))^{-1/2}P(t)P
%\\&
=PP(t)P+\frac{1}{2}PD(t)^2P+O((t-t_0)^3)\\
&=P-(t-t_0)^2 (P\dot{V}^t\big|_{t=t_0}S^2\dot{V}^t\big|_{t=t_0}P+ P\Theta_D   S^2\Theta_D  P\nonumber\\
&\quad-P\dot{V}^t\big|_{t=t_0} S^2\Theta_D  P- P\Theta_D S^2\dot{V}^t\big|_{t=t_0}P)+\frac{1}{2}(t-t_0)^2P(P^{(1)})^2P+\mathrm{o}(t-t_0)^2\\
&=P-(t-t_0)^2 (P\dot{V}^t\big|_{t=t_0}S^2\dot{V}^t\big|_{t=t_0}P+ P\Theta_D   S^2\Theta_D  P\nonumber\\
&\quad-P\dot{V}^t\big|_{t=t_0} S^2\Theta_D  P- P\Theta_D S^2\dot{V}^t\big|_{t=t_0}P)\\
&\quad+\frac{1}{2}(t-t_0)^2 (P\dot{ V}^t\big|_{t=t_0}S^2\dot{ V}^t\big|_{t=t_0}P-P\dot{ V}^t\big|_{t=t_0}S^2\Theta_D P-P\Theta_D S^2\dot{ V}^t\big|_{t=t_0}P\\
&\quad+ P\Theta_D S^2\Theta_D P)+\mathrm{o}(t-t_0)^2\\
&=P-\frac{1}{2}(t-t_0)^2 (P\dot{ V}^t\big|_{t=t_0}S^2\dot{ V}^t\big|_{t=t_0}P-P\dot{ V}^t\big|_{t=t_0}S^2\Theta_D P-P\Theta_D S^2\dot{ V}^t\big|_{t=t_0}P\\
&\quad+ P\Theta_D S^2\Theta_D P)+\mathrm{o}(t-t_0)^2,
\end{split}
\end{equation*}
and a similar argument yields \eqref{u-1}. Also, \eqref{7.44.1} follows using \eqref{ut-1}.

Next, let us consider the operator
\[ PU(t)^{-1}R(\zeta, t)U(t)P=A_1+A_2+A_3+A_4, \]
where we denote
\begin{align*}
\begin{split}
A_1&=PU(t)^{-1}PR(\zeta, t)PU(t)P,%\\&
\quad A_2=PU(t)^{-1}(I-P)R(\zeta, t)PU(t)P,\\
 A_3&=PU(t)^{-1}PR(\zeta, t)(I-P)U(t)P,
%\\&\quad\qquad
 \quad A_4=PU(t)^{-1}(I-P)R(\zeta, t)(I-P)U(t)P.
\end{split}
\end{align*}
Using \eqref{u-1}, \eqref{prp} and \eqref{pup},  we obtain
\begin{align*}
\begin{split}
A_1=&(P-\frac{1}{2}(t-t_0)^2(P\dot{ V}^t\big|_{t=t_0}S^2\dot{ V}^t\big|_{t=t_0}P-P\dot{ V}^t\big|_{t=t_0}S^2\Theta_D P-P\Theta_D S^2\dot{ V}^t\big|_{t=t_0}P\\
&+ P\Theta_D S^2\Theta_D P)+\mathrm{o}(t-t_0)^2\\
&\qquad\times((\lambda^{t_0}-\zeta)^{-1}P+(t-t_0) (\lambda^{t_0}-\zeta)^{-2}(-P\dot{ V}^t\big|_{t=t_0}+P\Theta_D )P\nonumber\\&+(t-t_0)^2(\lambda^{t_0}-\zeta)^{-2}(\tilde{R}_2(\zeta, t_0)-\frac{1}{2}P\ddot{V}^t\big|_{t=t_0}P)+\mathrm{o}(t-t_0)^{2}_u)\\
&\qquad\times(P-\frac{1}{2}(t-t_0)^2(P\dot{ V}^t\big|_{t=t_0}S^2\dot{ V}^t\big|_{t=t_0}P-P\dot{ V}^t\big|_{t=t_0}S^2\Theta_D P-P\Theta_D S^2\dot{ V}^t\big|_{t=t_0}P\\
&+ P\Theta_D S^2\Theta_D P)+\mathrm{o}(t-t_0)^2)\\
=&(\lambda^{t_0}-\zeta)^{-1}P+(\lambda^{t_0}-\zeta)^{-2}(t-t_0)(-P\dot{ V}^t\big|_{t=t_0}P+P\Theta_D P)-\frac{1}{2}(\lambda^{t_0}-\zeta)^{-2}(t-t_0)^2P\ddot{V}^t\big|_{t=t_0}P\\&+(\lambda^{t_0}-\zeta)^{-2}(t-t_0)^2\tilde{R}_2(\zeta, t_0)\\&-(\lambda^{t_0}-\zeta)^{-1}(t-t_0)^2(P\dot{ V}^t\big|_{t=t_0}S^2\dot{ V}^t\big|_{t=t_0}P-P\dot{ V}^t\big|_{t=t_0}S^2\Theta_D P-P\Theta_D S^2\dot{ V}^t\big|_{t=t_0}P\\
&+ P\Theta_D S^2\Theta_D P)+\mathrm{o}(t-t_0)^{2}_u.
\end{split}
\end{align*}
 Applying integration $-\frac{1}{2\pi i}\int_\gamma\zeta(\cdot)\,d\zeta$ yields
 \begin{align*}
 -\frac{1}{2\pi i}\int_\gamma\zeta A_1\,d\zeta=&\lambda^{t_0} P-(t-t_0)(-P\dot{ V}^t\big|_{t=t_0}P+P\Theta_D P)+\frac{1}{2}(t-t_0)^2P\ddot{V}^t\big|_{t=t_0}P\\
 +&(t-t_0)^2(-P\dot{V}^t\big|_{t=t_0}S\dot{V}^t\big|_{t=t_0}P- P\Theta_D   S\Theta_D  P+P\dot{V}^t\big|_{t=t_0} S\Theta_D  P+ P\Theta_D S\dot{V}^t\big|_{t=t_0}P)\\
 -&2\lambda^{t_0}(t-t_0)^2(P\dot{ V}^t\big|_{t=t_0}S^2\dot{ V}^t\big|_{t=t_0}P-P\dot{ V}^t\big|_{t=t_0}S^2\Theta_D P-P\Theta_D S^2\dot{ V}^t\big|_{t=t_0}P\\+& P\Theta_D S^2\Theta_D P)+\mathrm{o}(t-t_0)^{2}.
 \end{align*}

Also, it follows from \eqref{7.44.1} and \eqref{7.44.2} that 
\begin{align*}
\begin{split}
A_2=&((t-t_0) (-P\dot{ V}^t\big|_{t=t_0}S+P\Theta_D S)+\mathrm{o}(t-t_0))%\\&\times
((t-t_0)  (\lambda^{t_0}-\zeta)^{-1}(I-P)R({\zeta}, t_0)(-\dot{ V}^t\big|_{t=t_0}+\Theta_D) P\\\quad+&\mathrm{o}(t-t_0)_u)(P-\frac{1}{2}(t-t_0)^2 (P\dot{ V}^t\big|_{t=t_0}S^2\dot{ V}^t\big|_{t=t_0}P-P\dot{ V}^t\big|_{t=t_0}S^2\Theta_D P-P\Theta_D S^2\dot{ V}^t\big|_{t=t_0}P\\
&+ P\Theta_D S^2\Theta_D P)+\mathrm{o}(t-t_0)^2)\\
=&(\lambda^{t_0}-\zeta)^{-1}(t-t_0)^2 (-P\dot{ V}^t\big|_{t=t_0}S+P\Theta_D S)R({\zeta}, t_0)(-\dot{ V}^t\big|_{t=t_0}+\Theta_D) P+\mathrm{o}(t-t_0)^2_u.
\end{split}
\end{align*}
Applying integration $-\frac{1}{2\pi i}\int_\gamma\zeta(\cdot)\,d\zeta$ yields
 \begin{align*}
 -\frac{1}{2\pi i}\int_\gamma\zeta A_2\,d\zeta=&\lambda^{t_0}(t-t_0)^2 (-P\dot{ V}^t\big|_{t=t_0}S+P\Theta_D S)S(-\dot{ V}^t\big|_{t=t_0}+\Theta_D) P+\mathrm{o}(t-t_0)^2_u\\
 =&\lambda^{t_0}(t-t_0)^2(P\dot{ V}^t\big|_{t=t_0}S^2\dot{ V}^t\big|_{t=t_0}P-P\dot{ V}^t\big|_{t=t_0}S^2\Theta_D P-P\Theta_D S^2\dot{ V}^t\big|_{t=t_0}P\\ &\qquad+ P\Theta_D S^2\Theta_D P)+\mathrm{o}(t-t_0)^{2}.
 \end{align*}
Similarly,
\begin{align*}
\begin{split}
A_3&=PU(t)^{-1}PR(\zeta, t)(I-P)U(t)P\\
&=(P-\frac{1}{2}(t-t_0)^2 (P\dot{ V}^t\big|_{t=t_0}S^2\dot{ V}^t\big|_{t=t_0}P-P\dot{ V}^t\big|_{t=t_0}S^2\Theta_D P-P\Theta_D S^2\dot{ V}^t\big|_{t=t_0}P\\
&\quad+ P\Theta_D S^2\Theta_D P)+\mathrm{o}(t-t_0)^2)((t-t_0) (\lambda^{t_0}-\zeta)^{-1}(-P\dot{ V}^t\big|_{t=t_0}+ P\Theta_D) R(\zeta, t_0)+\mathrm{o}(t-t_0)_u)\\
&\qquad \times((t-t_0) ((-S\dot{ V}^t\big|_{t=t_0}+ S\Theta_D) P+\mathrm{o}(t-t_0))\\&=(\lambda^{t_0}-\zeta)^{-1}(t-t_0)^2 (-P\dot{ V}^t\big|_{t=t_0}+ P\Theta_D) R(\zeta, t_0)(-S\dot{ V}^t\big|_{t=t_0}+ S\Theta_D) P+\mathrm{o}(t-t_0)^2_u,\\
A_4&=PU(t)^{-1}(I-P)R(\zeta, t)(I-P)U(t)P\\
&=((t-t_0) (-P\dot{ V}^t\big|_{t=t_0}S+P\Theta_D S)+\mathrm{o}(t-t_0))(R({\zeta}, t_0)+\mathrm{o}(t-t_0)_u))\\&\qquad\times((t-t_0) ((-S\dot{ V}^t\big|_{t=t_0}+ S\Theta_D) P+\mathrm{o}(t-t_0))\\&=(t-t_0)^2(-P\dot{ V}^t\big|_{t=t_0}S+P\Theta_D S)R({\zeta}, t_0)(-S\dot{ V}^t\big|_{t=t_0}+ S\Theta_D) P+\mathrm{o}(t-t_0)^2_u.
\end{split}
\end{align*}
Therefore,
\begin{align*}
 -\frac{1}{2\pi i}\int_\gamma\zeta A_3\,d\zeta=&\lambda^{t_0}(t-t_0)^2(P\dot{ V}^t\big|_{t=t_0}S^2\dot{ V}^t\big|_{t=t_0}P-P\dot{ V}^t\big|_{t=t_0}S^2\Theta_D P-P\Theta_D S^2\dot{ V}^t\big|_{t=t_0}P\\+& P\Theta_D S^2\Theta_D P)+\mathrm{o}(t-t_0)^{2},\\
 -\frac{1}{2\pi i}\int_\gamma\zeta A_4\,d\zeta=&\mathrm{o}(t-t_0)^2.
 \end{align*}
Collecting all these terms and using the standard relation from  \cite[Equation (III.6.24)]{Kato}
\[L^t_{\cG}P(t)=-\frac{1}{2\pi i}\int_\gamma \zeta R(\zeta,t)\,d\zeta\]
  yields \eqref{pulup} and concludes the proof.

%\begin{theorem}[Second derivatives]
%	Assume Hypothesis \ref{h1} (i)-(iii) and (iv'). Let $\cG$ be the Dirichlet subspace, or a Neumann-based subspace of the boundary space $\cH=H^{1/2}(\dOm)\times H^{-1/2}(\dOm)$. Let $\lambda(t_0)$ be a give eigenvalue of multiplicity $m$ of the operator $L_{\cG}^{\Om_{t_0}}$ defined in \eqref{L(Omt)}, and let $\{\lambda_{jk}(t)\}$ denote the eigenvalues of the operator $L_{\cG}^{\Om_t}$ for $t$ near $t_0$. Let us consider a normalized basis $\{u_{jk}^{t_0}\}\in\ran(P)$ of the eigenvectors of the operator $P^{(1)}_j T^{(2)} P^{(1)}_j$ defined in \eqref{T2} corresponding to its eigenvalues $\lambda^{(2)}_{jk}$ so that $T^{(2)}u_{jk}^{t_0}=\lambda_{jk}^{(2)}u_{jk}^{t_0}$. We denote by $q_{jk}=(
%		\gaD u_{jk}^{t_0}, \frac{1}{t_0}\gaN u_{jk}^{t_0})\in\Upsilon(t_0)\cap\cG$ the respective rescaled traces. Then the second derivatives of $\lambda_{jk}(t)$ are given by the formulas
		
%		\begin{equation}
	%		\frac{d^2\lambda_{jk}(t)}{dt^2}\Big|_{t=t_0}=\frac1{t_0}\dot{\mathfrak{m}}_{t_0}(q_{jk},q_{jk})-\frac1{t^2_0}{\mathfrak{m}}_{t_0}(q_{jk},q_{jk}),
	%		\end{equation}
	%	where the Maslov form and its derivative for the flow $\Upsilon(t)=\tr_t(\cK_t)$ can be computed as follows:
%	\begin{equation}
%	\mathfrak{m}_{t_0}(q_{jk},q_{jk})=\frac1{t_0}(t^{(1)}(u_{jk}^{t_0},u_{jk}^{t_0})-\langle 2t_0\lambda_{jk}(t_0)u_{jk}^{t_0},u_{jk}^{t_0}\rangle_{L^2(\Om)}).
%	\end{equation}
	
%\end{theorem}
%%%%%%%%%%%%%%%%%%%%%%%%%%%%%%%%%%%%%%%

\end{document}